\documentclass[11pt]{article}
\usepackage[margin=1in]{geometry}
\usepackage{amsmath,amssymb,amsthm,mathtools}
\usepackage{enumitem}
\usepackage{hyperref}
\usepackage{tikz-cd}
\usepackage{bm}
\usetikzlibrary{arrows.meta}

\theoremstyle{plain}
\newtheorem{theorem}{Theorem}[section]
\newtheorem{lemma}[theorem]{Lemma}
\newtheorem{proposition}[theorem]{Proposition}
\newtheorem{corollary}[theorem]{Corollary}
\theoremstyle{definition}
\newtheorem{definition}[theorem]{Definition}

\newtheorem{example}[theorem]{Example}
\newtheorem{remark}[theorem]{Remark}

\newcommand{\kk}{k}
\newcommand{\Zn}{\mathbb{Z}^n}
\newcommand{\NN}{\mathbb{N}}
\newcommand{\ZZ}{\mathbb{Z}}
\newcommand{\RR}{\mathbb{R}}
\newcommand{\pp}{\mathfrak{p}}
\newcommand{\mm}{\mathfrak{m}}

\newcommand{\link}{\operatorname{lk}}
\newcommand{\del}{\operatorname{del}}
\newcommand{\st}{\operatorname{star}}
\newcommand{\supp}{\operatorname{supp}}

\newcommand{\Ext}{\operatorname{Ext}}
\newcommand{\rank}{\operatorname{rank}}

\newcommand{\rH}{\widetilde{H}}
\newcommand{\Hloc}[2]{H^{#1}_{#2}}

\title{Localized Persistent Commutative Algebra}

\author{Kaiyue He\thanks{Syracuse University \texttt{hekaiyue@hotmail.com}}
  \and  Faisal Suwayyid \thanks{King Fahd University of Petroleum and Minerals \texttt{faisal.suwayyid@kfupm.edu.sa}}
  \and 
  Guo-Wei Wei\thanks{University of Georgia \texttt{Guowei.Wei@uga.edu}}
}

\date{}

\begin{document}
\maketitle

\begin{abstract}
We develop a localized persistent theory of commutative algebra via the local cohomology for Stanley--Reisner rings. 
The theory is modeled on the Suwayyid--Wei persistence of the Stanley--Reisner theory
\cite{SuwayyidWeiPSRT} and its functorial development for graphs and hypergraphs
\cite{SuwayyidWeiGraphs}, in which one persists invariants of the face ring across a filtration, such as its graded Betti numbers and $f$- and $h$-vectors. That framework is built from the minimal free
resolution and is thus Tor-theoretic; we work instead on the injective side, persisting invariants that
arise from local cohomology, and the resulting per-vertex modules record information localized at an
individual vertex, complementing the global picture recorded by maximal-support local cohomology. The
construction rests on an exact, multigraded decomposition (Theorem~\ref{thm:structure}) that reduces
local cohomology at a vertex prime to ordinary (maximal-support) local cohomology of the link and the
deletion of the vertex; combined with Hochster's formula, it yields a closed combinatorial description
of every multigraded piece (Corollary~\ref{cor:static-comb}). At the level of graded dimensions, this decomposition is the vertex-prime case of Rahimi's bigraded local-cohomology formula \cite{Rahimi}. We refine this description by identifying the degree-zero and positive-$x_i$-degree pieces with the
maximal-support local cohomology of the deletion and link, respectively. Building on this structure, we
introduce per-vertex persistent local-cohomology numbers, a persistent links--Hochster formula,
interval decompositions of the resulting reversed-arrow persistence modules, and a bottleneck-stability
theorem. 
\end{abstract}

\tableofcontents

\section{Motivation, setup, and notation}\label{sec:setup}
 
The introduction of persistent Stanley--Reisner theory by Suwayyid and Wei \cite{SuwayyidWeiPSRT} has enabled successful applications of commutative algebra to machine learning, data science, and biology \cite{suwayyid2025cakr, wei2026commutative, zhang2026commutative, feng2025caml, zia2025gbnl}. This framework provides powerful algebraic tools, including persistent facet ideals, persistent $f$- and $h$-vectors, and persistent graded Betti numbers. The connections  between persistent commutative algebra and traditional topological data analysis \cite{carlsson2009topology}  have also been investigated \cite{ren2025interpretability, hu2025commutative}. In this sense, commutative algebra provides a distinctive collection of algebraic and topological tools for data science that is not readily accessible through other mathematical, statistical, physical, or engineering approaches.

However, the Suwayyid--Wei persistence framework for Stanley--Reisner theory primarily captures global algebraic invariants of data, such as the graded Betti numbers, $f$- and $h$ vectors. Therefore is not well suited for local analysis, which is important for wide class of applications in data science. For example, one may wants to understand the function of an individual atom in a molecule or an individual residue in a protein. 

To overcome this limitation, we introduce a localized persistent commutative algebra theory. The theory is motivated by an algebraic technique : localizing the ring at prime ideals. The local behavior of a ring $R$ near a point in $\operatorname{Spec} R$ can be captured by the ring $R_{\pp}$ and the derived functor $\Hloc{\bullet}{\pp}(-)$ from the $\pp$-torsion functor. For a Stanley--Reisner ring this mechanism has a combinatorial
interpretation. Under the standing convention that every singleton $\{i\}$ is a face,
the vertex prime $\pp_i=(x_j:j\neq i)$ contains $I_\Delta$, so it descends to a prime
of $\kk[\Delta]$ with
\[
\kk[\Delta]/\pp_i\kk[\Delta]\;\cong\;\kk[x_i],
\]
the coordinate line through the vertex $i$. The assignment $i\mapsto\pp_i$ therefore
realizes the vertex set of $\Delta$ inside $\operatorname{Spec}\kk[\Delta]$, and
localizing at $\pp_i$ is precisely the algebraic operation of restricting attention to
the vertex $i$. This is the localization we persist.

The similar distinction between global and local invariants also appears in the topological side. Ordinary persistent homology of a filtration is a global summary,
and local homology $H_\bullet(X,X\setminus x)$ was introduced into topological data
analysis for exactly the complementary purpose of probing a space at a chosen point,
in particular for the recovery of stratifications from samples
\cite{BendichCSEHM,BendichWangMukherjee}. For a simplicial complex, local homology at
a vertex is the reduced homology of its link \cite[\S63]{Munkres}, and it is precisely
the link that Theorem~\ref{thm:structure} produces as one of the two summands of
$\Hloc{\bullet}{\pp_i}(\kk[\Delta])$; \S\ref{sec:scope} makes this comparison precise.
The construction of this paper may thus be read in two ways: as the injective-side,
per-vertex counterpart of the Tor-side theory of
\cite{SuwayyidWeiPSRT,SuwayyidWeiGraphs}, and as an algebraic refinement of persistent
local homology in which the deletion summand and the $x_i$-action carry data that the
purely topological invariant does not record.

We follow the conventions of 
 Suwayyid-Wei persistence \cite{SuwayyidWeiPSRT} and its
functorial development for graphs and hypergraphs \cite{SuwayyidWeiGraphs}. Let $\kk$ be a field and
$S=\kk[x_1,\dots,x_n]$ the polynomial ring with its standard $\ZZ$-grading $\deg x_i=1$ and its
finer $\Zn$-grading $\deg x_i=\mathbf e_i$. Let $\Delta$ be an abstract simplicial complex on the
vertex set $[n]=\{1,\dots,n\}$, with the convention that every singleton $\{i\}$ is a face. The
\emph{Stanley--Reisner ideal} is
\[
I_\Delta=\big(x_{i_1}\cdots x_{i_r}:\{i_1,\dots,i_r\}\notin\Delta\big)\subseteq S,
\qquad
\kk[\Delta]=S/I_\Delta .
\]
For $W\subseteq[n]$ write $\Delta_W=\{\sigma\in\Delta:\sigma\subseteq W\}$ for the induced
subcomplex, and $\mathbf 1_W=\sum_{i\in W}\mathbf e_i\in\{0,1\}^n$ for its indicator multidegree.

For a face $\sigma\in\Delta$ we use the three standard local operations:
\begin{align*}
\link_\Delta(\sigma)&=\{\tau\in\Delta:\tau\cap\sigma=\varnothing,\ \tau\cup\sigma\in\Delta\}
&&\text{(link)},\\
\del_\Delta(\sigma)&=\{\tau\in\Delta:\tau\cap\sigma=\varnothing\}
&&\text{(deletion)},\\
\st_\Delta(\sigma)&=\{\tau\in\Delta:\tau\cup\sigma\in\Delta\}
&&\text{(closed star)} .
\end{align*}
We will only need these for a single vertex $\sigma=\{i\}$, abbreviated $\link_\Delta(i)$,
$\del_\Delta(i)$. Note $\del_\Delta(i)$ is the induced subcomplex
$\Delta_{[n]\setminus\{i\}}$, while
$\link_\Delta(i)=\{\tau\subseteq[n]\setminus\{i\}:
\tau\cup\{i\}\in\Delta\}$. Both are complexes whose
vertex sets are contained in $[n]\setminus\{i\}$, and
$\link_\Delta(i)\subseteq\del_\Delta(i)$.

Fix a vertex $i$. We write $S'=\kk[x_j:j\neq i]$ for the polynomial ring in the remaining variables, with maximal
graded ideal $\mm'=(x_j:j\neq i)$. The Stanley--Reisner rings $\kk[\del_\Delta(i)]$ and
$\kk[\link_\Delta(i)]$ are quotients of $S'$. For a finitely generated $\ZZ^m$-graded module $M$
over a polynomial ring and a monomial ideal $\mathfrak a$, $\Hloc{\bullet}{\mathfrak a}(M)$ denotes
local cohomology with support in $\mathfrak a$; we recall its computation by the (stable Koszul,
or \v{C}ech) complex below. Reduced simplicial (co)homology with coefficients in $\kk$ is written
$\rH_\bullet(-;\kk)$, $\rH^\bullet(-;\kk)$; over a field these are dual and have equal dimensions.

\paragraph{The vertex prime.} The object of study is local cohomology supported at the
\emph{vertex prime}
\[
\pp_i:=(x_j:j\neq i)\subseteq S,\qquad \dim S/\pp_i=1,
\]
the defining ideal of the coordinate line through vertex $i$. Note $\pp_i$ is generated by
all variables except $x_i$, and $\pp_i$ is the extension to $S$ of the
maximal ideal $\mm'$ of $S'$; equivalently $S=S'[x_i]$ and $\pp_i=\mm' S$.

\section{Local cohomology at a vertex prime: the static theory}\label{sec:static}

\subsection{The multigraded \v{C}ech complex}\label{ssec:cech}

We recall the standard computation of local cohomology with monomial support in the
$\Zn$-graded form we use. Let $T\subseteq[n]$ and $\mathfrak a_T=(x_j:j\in T)$. The (cohomological)
\v{C}ech complex of an $S$-module $M$ with respect to the generators $\{x_j\}_{j\in T}$ is
\begin{equation}\label{eq:cech}
\check C^\bullet_T(M):\quad
0\to M\to\bigoplus_{j\in T} M_{x_j}\to
\bigoplus_{\substack{F\subseteq T\\|F|=2}} M_{x_F}\to\cdots\to M_{x_T}\to 0,
\end{equation}
with $x_F=\prod_{j\in F}x_j$; the term in cohomological degree $p$ is
$\bigoplus_{F\subseteq T,\,|F|=p}M_{x_F}$. The differential is the alternating sum of the canonical maps between the
localization summands. We use the natural ordering $1<\cdots<n$. For
$F\subseteq T$ and $j\in T\setminus F$, set
\(
\nu_F(j):=\bigl|\{u\in F:u<j\}\bigr|.
\)
On the summand $M_{x_F}$ in cohomological degree $p=|F|$, the differential is
defined by
\[
d^p\left(\frac{m}{x_F^N}\right)
=
\sum_{j\in T\setminus F}
(-1)^{\nu_F(j)}
\frac{m x_j^N}{x_{F\cup\{j\}}^N},
\]
where the term indexed by $j$ is placed in the summand
$M_{x_{F\cup\{j\}}}$ of cohomological degree $p+1$.
Thus each map from $M_{x_F}$ to $M_{x_{F\cup\{j\}}}$ has the fixed sign
$(-1)^{\nu_F(j)}$, while the signs alternate among the subsets obtained by
removing one element from a fixed set $G$. We use the conventions
$x_\varnothing=1$ and $M_{x_\varnothing}=M$. 

Here $M_{x_F}$ denotes the localization of $M$ at the
multiplicative set $\{1,x_F,x_F^2,\dots\}$; equivalently,
$M_{x_F}=M\otimes_S S_{x_F}$. It is classical
\cite[\S3.5]{BrunsHerzog}, \cite[\S A1.4]{Eisenbud} that
$\Hloc{q}{\mathfrak a_T}(M)\cong H^q(\check C^\bullet_T(M))$ for $q\geq0$; for local cohomology with
monomial support specifically, see also \cite{EisenbudMustataStillman}.
When $M$ is $\Zn$-graded and the $x_j$ are homogeneous, each $M_{x_F}$ is $\Zn$-graded and the
differentials are degree-preserving, so $\Hloc{q}{\mathfrak a_T}(M)$ is $\Zn$-graded and may be
computed one multidegree at a time.

For $M=\kk[\Delta]$ the localizations have a transparent monomial description.

\begin{lemma}[Graded pieces of localizations of a Stanley--Reisner ring]\label{lem:localization}
Let $F\subseteq[n]$. For $a\in\Zn$, the monomial $x^a := \prod x_i^{a_i}$ represents a nonzero element of
$\kk[\Delta]_{x_F}$ in multidegree $a$ if and only if
\begin{enumerate}[label=\textup{(\roman*)},nosep]
\item $a_j\geq 0$ for every $j\notin F$, and
\item $\supp_+(a)\cup F\in\Delta$, where $\supp_+(a)=\{j:a_j>0\}$.
\end{enumerate}
In that case $\big(\kk[\Delta]_{x_F}\big)_a=\kk\cdot x^a$, and otherwise it is $0$. In particular
$\kk[\Delta]_{x_F}\neq 0$ iff $F\in\Delta$.
\end{lemma}

\begin{proof}
This is the standard $\Zn$-graded description of the localizations of a Stanley--Reisner ring; see
\cite[Lemma~5.3.6]{BrunsHerzog}, and \cite[Lemma~1.1]{Rahimi} for the bigraded form used below.
In outline, $\kk[\Delta]_{x_F}=\varinjlim_N x_F^{-N}\kk[\Delta]$ along multiplication by $x_F$, so a class
in multidegree $a$ is represented by $x^{a+N\mathbf 1_F}$ for $N\gg0$ and is nonzero iff
$a+N\mathbf 1_F\in\NN^n$ and $\supp(a+N\mathbf 1_F)\in\Delta$ for large $N$. Adding multiples of
$\mathbf 1_F$ changes no coordinate outside $F$, so the first condition is (i), and
$\supp(a+N\mathbf 1_F)=\supp_+(a)\cup F$ for large $N$ gives (ii). The graded piece is then at most
one-dimensional, spanned by $x^a$.
\end{proof}

Next we give an example illustrating the previous lemma.

\begin{example}\label{ex:localization}
Let $\Delta$ be the path $x_1\!-\!x_2\!-\!x_3$ on $\{1,2,3\}$, with facets $\{1,2\}$ and $\{2,3\}$.
Its only non-faces are $\{1,3\}$ and $\{1,2,3\}$, so $I_\Delta=(x_1x_3)$ and
$\kk[\Delta]=\kk[x_1,x_2,x_3]/(x_1x_3)$. We test when the class of $x^a$ is nonzero in various
localizations; recall the criterion is (i)~$a_j\ge0$ for $j\notin F$, and (ii)~$\supp_+(a)\cup F\in\Delta$.

\emph{Inverting $x_2$ only \textup{(}$F=\{2\}$, and $\{2\}\in\Delta$\textup{)}.} Here $\kk[\Delta]_{x_2}\neq0$.
\begin{itemize}[nosep]
\item $a=(1,0,0)$, i.e.\ $x^a=x_1$: condition (i) holds ($a_1,a_3\ge0$), and $\supp_+(a)\cup F=\{1,2\}\in\Delta$,
so the class is \emph{nonzero}.
\item $a=(1,0,1)$, i.e.\ $x^a=x_1x_3$: condition (i) holds, but $\supp_+(a)\cup F=\{1,2,3\}\notin\Delta$,
so the class is \emph{zero}---consistent with $x_1x_3=0$ in $\kk[\Delta]$.
\item $a=(0,-2,1)$, i.e.\ $x^a=x_3x_2^{-2}$: the negative exponent is on $x_2\in F$, (i) holds,
and $\supp_+(a)\cup F=\{2,3\}\in\Delta$, so the class is \emph{nonzero}.
\end{itemize}

\emph{Inverting $x_1$ $x_3$ \textup{(}$F=\{1,3\}$\textup{)}.} Since $\{1,3\}\notin\Delta$ we have
$x_1x_3=0$ in $\kk[\Delta]$, so we are inverting a nilpotent (indeed zero) element: the localization
$\kk[\Delta]_{x_1x_3}$ is the \emph{zero ring}.

\emph{Inverting $x_1$ $x_2$ \textup{(}$F=\{1,2\}\in\Delta$\textup{)}.} Here $\kk[\Delta]_{x_1x_2}\neq0$.
The class of $x^a$ with $a=(-3,-1,0)$ is nonzero (negatives only on $F$; $\supp_+(a)\cup F=\{1,2\}\in\Delta$),
whereas for $a=(-3,-1,1)$ it is zero, since $\supp_+(a)\cup F=\{1,2,3\}\notin\Delta$. Thus even after
inverting $x_1,x_2$, appending a positive power of $x_3$ kills the class.
\end{example}

\subsection{The splitting along a vertex}\label{ssec:split}

The \v{C}ech complex computing $\Hloc{\bullet}{\pp_i}$ inverts only the variables $x_j$, $j\neq i$;
hence it is a complex of $S'$-modules whose differentials preserve the $x_i$-degree. Write
$T=[n]\setminus\{i\}$, so $\pp_i=\mathfrak a_T$. Every face of $\Delta$ either omits $i$ or contains
$i$, so every monomial of $\kk[\Delta]$ has $x_i$-exponent (the coordinate $a_i$
 of its multidegree $a=(a_1,\dots,a_n)$) equals
 $0$ or $\geq1$, giving the
$\Zn$-graded $\kk$-vector space decomposition
\begin{equation}\label{eq:ring-split}
\kk[\Delta]\;=\;\underbrace{\kk[\del_\Delta(i)]}_{\text{monomials with }a_i=0}\ \oplus\
\underbrace{\bigoplus_{c\geq1} x_i^{\,c}\cdot\kk[\link_\Delta(i)]}_{\text{monomials with }a_i\geq 1}.
\end{equation}

\begin{lemma}[Chain-level splitting of the \v{C}ech complex]\label{lem:chain-split}
Let $T=[n]\setminus\{i\}$. As a complex of $\Zn$-graded $\kk$-vector spaces,
$\check C^\bullet_T(\kk[\Delta])$ splits as a direct sum indexed by the $x_i$-exponent $c\in\ZZ$
(the coordinate $a_i$ of the multidegree):
\[
\check C^\bullet_T(\kk[\Delta])\;=\;\bigoplus_{c\in\ZZ}\ \check C^\bullet_T(\kk[\Delta])^{(c)},
\]
where $\check C^\bullet_T(\kk[\Delta])^{(c)}$ is the $\kk$-span of the basis monomials $x^a$ with
$a_i=c$. Writing $\check C^\bullet_{\mm'}(-)$ for the \v{C}ech complex over $S'$ with respect to
$\mm'=(x_j:j\neq i)$, there are isomorphisms of complexes of $\Zn$-graded $\kk$-vector spaces
\[
\check C^\bullet_T(\kk[\Delta])^{(c)}\;\cong\;
\begin{cases}
\check C^\bullet_{\mm'}\!\big(\kk[\del_\Delta(i)]\big), & c=0,\\[2pt]
x_i^{\,c}\cdot \check C^\bullet_{\mm'}\!\big(\kk[\link_\Delta(i)]\big), & c\geq1,\\[2pt]
0, & c\leq -1 .
\end{cases}
\]
\end{lemma}

\begin{proof}
The \v{C}ech complex $\check C^\bullet_T(\kk[\Delta])$ inverts only the variables $x_j$, $j\neq i$, so its
differentials preserve the coordinate $a_i$; this is the bigrading of \cite[\S1]{Rahimi} for the vertex
prime, the first factor being the $x_i$-degree $c$. Hence the complex is the direct sum of its
$x_i$-degree-$c$ subcomplexes $\check C^\bullet_T(\kk[\Delta])^{(c)}$, and by
Lemma~\ref{lem:localization} every basis monomial of every term has $a_i\ge0$, so the summands with
$c\le-1$ are $0$. To identify the remaining pieces, fix $c\ge0$, $b\in\ZZ^{T}$, and $a=(c;b)$: by
Lemma~\ref{lem:localization}, $x^a$ is nonzero in $\kk[\Delta]_{x_F}$ ($F\subseteq T$) iff $b_j\ge0$ for
$j\in T\setminus F$ and $\supp_+(a)\cup F\in\Delta$. Since $\supp_+(a)=\supp_+(b)$ for $c=0$ and
$\supp_+(a)=\supp_+(b)\cup\{i\}$ for $c\ge1$, this reads $\supp_+(b)\cup F\in\del_\Delta(i)$ (for $c=0$),
resp.\ $\supp_+(b)\cup F\in\link_\Delta(i)$ (for $c\ge1$)---which, again by Lemma~\ref{lem:localization}
over $S'$, is exactly the condition for $x^b$ nonzero in $\kk[\del_\Delta(i)]_{x_F}$, resp.\
$\kk[\link_\Delta(i)]_{x_F}$. The bijection $x^a\mapsto x^b$ ($c=0$), resp.\ $x^a\mapsto x_i^{\,c}\cdot x^b$
($c\ge1$), commutes with the identical \v{C}ech differentials, which gives the stated isomorphisms of
complexes.
\end{proof}

\begin{remark}
Decomposition \eqref{eq:ring-split} is the $\Zn$-graded shadow of the classical bigrading of
$\kk[\Delta]$ along a monomial prime used to study $\Hloc{\bullet}{\pp_i}$
\cite{Rahimi}; Lemma~\ref{lem:chain-split} makes the \v{C}ech-level consequence explicit for
the vertex prime. It is parallel to Hochster's decomposition of $\Hloc{\bullet}{\mm}(\kk[\Delta])$
\cite[\S II.4]{StanleyCCA} and the poset/sheaf splittings of Brun--Bruns--R\"omer
\cite{BrunBrunsRomer}.
\end{remark}

\subsection{The structure theorem}\label{ssec:structure}

Taking cohomology in Lemma~\ref{lem:chain-split} gives the central structural result.

\begin{theorem}[Structure of vertex-prime local cohomology]\label{thm:structure}
Let $\Delta$ be a simplicial complex on $[n]$ and $i\in[n]$. For every $q\geq0$ there is an
isomorphism of $\Zn$-graded $S'$-modules
\[
\Hloc{q}{\pp_i}\!\big(\kk[\Delta]\big)\;\cong\;
\underbrace{\Hloc{q}{\mm'}\!\big(\kk[\del_\Delta(i)]\big)}_{x_i\text{-degree }0}\ \oplus\
\Big(\underbrace{\Hloc{q}{\mm'}\!\big(\kk[\link_\Delta(i)]\big)}_{\text{repeated in each }x_i\text{-degree}\,\geq1}
\otimes_\kk\, x_i\,\kk[x_i]\Big),
\]
where $S'=\kk[x_j:j\neq i]$ has maximal graded ideal $\mm'=(x_j:j\neq i)$, and
$\del_\Delta(i)=\Delta_{[n]\setminus\{i\}}$, $\link_\Delta(i)$ are the deletion and link of the vertex
$i$, regarded as complexes over $S'$. The factor $x_i\,\kk[x_i]=\bigoplus_{c\geq1}\kk x_i^c$ records the
$x_i$-grading of the second summand. 

Equivalently, in a fixed multidegree $a=(c;b)$ with $c=a_i\in\ZZ$ and $b\in\ZZ^{[n]\setminus\{i\}}$,
\[
\Hloc{q}{\pp_i}\!\big(\kk[\Delta]\big)_a\;\cong\;
\begin{cases}
\Hloc{q}{\mm'}\!\big(\kk[\del_\Delta(i)]\big)_b, & c=0,\\[2pt]
\Hloc{q}{\mm'}\!\big(\kk[\link_\Delta(i)]\big)_b, & c\geq1,\\[2pt]
0, & c\leq-1 .
\end{cases}
\]
\end{theorem}

\begin{proof}
Local cohomology is the cohomology of the \v{C}ech complex, and cohomology commutes with the direct
sum decomposition of Lemma~\ref{lem:chain-split} (a finite direct sum in each multidegree). The
summand $\check C^\bullet_T(\kk[\Delta])^{(0)}$ (the multidegrees $a$ with $a_i=0$) equals
$\check C^\bullet_{\mm'}(\kk[\del_\Delta(i)])$, whose cohomology is
$\Hloc{q}{\mm'}(\kk[\del_\Delta(i)])$; for each $c\geq1$ the summand
$\check C^\bullet_T(\kk[\Delta])^{(c)}$ (the multidegrees with $a_i=c$) is the degree shift
$x_i^{\,c}\cdot\check C^\bullet_{\mm'}(\kk[\link_\Delta(i)])$, whose cohomology is
$\Hloc{q}{\mm'}(\kk[\link_\Delta(i)])$; the summands with $c\leq-1$ vanish. The chain-level
isomorphisms of Lemma~\ref{lem:chain-split} are $S'$-linear: multiplication by $x_j$ ($j\neq i$)
preserves the $x_i$-degree $c$ and commutes with the bijections $x^a\mapsto x^b$ ($c=0$) and
$x^a\mapsto x_i^{\,c}x^b$ ($c\ge1$), so the induced isomorphism on cohomology is one of $\Zn$-graded
$S'$-modules. Combining over all $c$ gives the stated isomorphism, with the link summand appearing
once for each $c\geq1$, i.e.\ tensored with $x_i\kk[x_i]$.
\end{proof}

\begin{example}\label{ex:two-summands}
Let $\Delta$ be the complex on $\{0,1,2,3\}$ with facets $\{0,1\}$,
$\{0,2\}$, and the isolated vertex $\{3\}$ (Figure~\ref{fig:two-summands}).
Take $i=0$; then
$\link_\Delta(0)=\{\varnothing,\{1\},\{2\}\}$ is two points,
regarded as a complex over $S'=\kk[x_1,x_2,x_3]$ with $3$ as a ghost
vertex, while
$\del_\Delta(0)=\{\varnothing,\{1\},\{2\},\{3\}\}$ is three points over
$S'$.

\begin{figure}[h]
\centering
\begin{tikzpicture}[scale=1.5,
    vtx/.style={circle,fill=black,inner sep=1.7pt},
    vlbl/.style={font=\small}]
\coordinate (v1) at (-1.3,0);
\coordinate (v0) at (0,0);
\coordinate (v2) at (1.3,0);
\coordinate (v3) at (2.7,0);
\draw[thick] (v1)--(v0)--(v2);
\foreach \p in {v0,v1,v2,v3} \node[vtx] at (\p) {};
\node[vlbl,above] at (v0) {$0$};
\node[vlbl,above] at (v1) {$1$};
\node[vlbl,above] at (v2) {$2$};
\node[vlbl,above] at (v3) {$3$};
\end{tikzpicture}
\caption{The complex $\Delta$ of Example~\ref{ex:two-summands}.}
\label{fig:two-summands}
\end{figure}
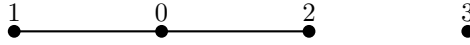

\emph{Local cohomology of $d$ isolated points.} For $\Gamma$ a set of $d$ isolated points,
$\kk[\Gamma]=\kk[x_1,\dots,x_d]/(x_jx_k:j<k)$ is one-dimensional Cohen--Macaulay, so only
$\Hloc{1}{\mm'}$ is nonzero, and Theorem~\ref{thm:hochster-classical} reads its graded pieces off the
links of faces: at $\sigma=\varnothing$ (link $=\Gamma$) it is $\dim_\kk\rH^{0}(\Gamma;\kk)=d-1$ in
multidegree $0$; at $\sigma=\{j\}$ (link $=\{\varnothing\}$) it is
$\dim_\kk\rH^{-1}(\{\varnothing\};\kk)=1$ in each multidegree $-m\mathbf e_j$ ($m\ge1$); all other pieces
vanish. Thus $\Hloc{1}{\mm'}(\kk[\del_\Delta(0)])$ has dimension $2$ at $(0,0,0)$ and $1$ at each
$-m\mathbf e_j$ ($j\in\{1,2,3\}$), while $\Hloc{1}{\mm'}(\kk[\link_\Delta(0)])$ has dimension $1$ at
$(0,0,0)$ and at each $-m\mathbf e_j$ ($j\in\{1,2\}$), and vanishes
in every multidegree with nonzero $x_3$-coordinate.

\emph{A representative \v{C}ech calculation.} The only faces of $\Delta$ inside $T=\{1,2,3\}$ are the
three vertices, so by Lemma~\ref{lem:localization} $\kk[\Delta]_{x_F}=0$ for $|F|\ge2$ and
\eqref{eq:cech} is the two-term complex
\[
0\to\kk[\Delta]\xrightarrow{\ d\ }\kk[\Delta]_{x_1}\oplus\kk[\Delta]_{x_2}\oplus\kk[\Delta]_{x_3}\to0,
\qquad d(u)=(u,u,u),
\]
in cohomological degrees $0,1$. The associated primes of $\kk[\Delta]$ are $(x_2,x_3)$, $(x_1,x_3)$, and $(x_0,x_1,x_2)$, none of which contains $\pp_0=(x_1,x_2,x_3)$, so $\Gamma_{\pp_0}(\kk[\Delta])=0$, $d$ is injective, and only $\Hloc{1}{\pp_0}(\kk[\Delta])=\operatorname{coker}d$ survives. By Lemma~\ref{lem:localization} the only
face containing $3$ is $\{3\}$, so $\kk[\Delta]_{x_3}$ is supported in the multidegrees $(0,0,0,a_3)$
alone---in particular at $a_0=0$---whereas $\kk[\Delta]_{x_1},\kk[\Delta]_{x_2}$ are supported at all
$a_0\ge0$. Hence in each $x_0$-degree $c\ge1$ the $x_3$-summand drops out and $\operatorname{coker}d$ is
computed from $\kk[\Delta]_{x_1}\oplus\kk[\Delta]_{x_2}$, i.e.\ from the two-point link, while in
$x_0$-degree $0$ all three summands contribute, i.e.\ the three-point deletion. Concretely, at
$\alpha=(0,0,0,0)$ one has $d(1)=(1,1,1)$ in $\kk^3$, so
$\Hloc{1}{\pp_0}(\kk[\Delta])_{(0,0,0,0)}\cong\kk^2$; at $\alpha=(c,0,0,0)$ with $c\ge1$ only two
summands appear and $d(x_0^c)=(x_0^c,x_0^c)$, so $\Hloc{1}{\pp_0}(\kk[\Delta])_{(c,0,0,0)}\cong\kk$.

These match the graded dimensions above: $\Hloc{1}{\pp_0}(\kk[\Delta])$ agrees with
$\Hloc{1}{\mm'}(\kk[\del_\Delta(0)])$ in $x_0$-degree $0$ and with
$\Hloc{1}{\mm'}(\kk[\link_\Delta(0)])$ in each $x_0$-degree $c\ge1$, which is the splitting of
Theorem~\ref{thm:structure},
\[
\Hloc{1}{\pp_0}(\kk[\Delta])\cong\Hloc{1}{\mm'}(\kk[\del_\Delta(0)])\oplus
\big(\Hloc{1}{\mm'}(\kk[\link_\Delta(0)])\otimes_\kk x_0\kk[x_0]\big).
\]
\end{example}

\begin{remark}[The two summands]\label{rem:reading}
The two summands carry complementary information. The deletion summand
$\Hloc{q}{\mm'}(\kk[\del_\Delta(i)])$ is independent of how $i$ attaches to the rest of $\Delta$; it is
the local cohomology of the complex obtained from $\Delta$ by deleting the vertex $i$. The link summand
$\Hloc{q}{\mm'}(\kk[\link_\Delta(i)])$ depends only on the faces $\sigma$ of $\Delta$ with
$\sigma\cup\{i\}\in\Delta$, that is, on the structure of $\Delta$ at the vertex $i$. Thus the $i$-local
content of $\Hloc{\bullet}{\pp_i}(\kk[\Delta])$ is carried by the link summand, in accordance with the
fact that $\pp_i$ is the defining ideal of the coordinate line through vertex $i$.
\end{remark}

\subsection{The combinatorial (Hochster) dictionary}\label{ssec:hochster}

Theorem~\ref{thm:structure} reduces everything to maximal-support local cohomology of the link and
deletion, for which Hochster's classical formula gives a complete combinatorial description. We state
it in the precise $\Zn$-graded, link-indexed form we need and include a short self-contained proof,
both because the exact form matters below and because the cochain-level isomorphism in the proof is
what makes the persistent version (Theorem~\ref{thm:persistent-hochster}) natural; we then compare
with the standard references in Remark~\ref{rem:hochster-source}.

\begin{theorem}[Hochster's local-cohomology formula, a special version of
{\cite[Theorem~2 and Lemma, p.~273]{Grabe}}.]
\label{thm:hochster-classical}
Let $\Gamma$ be a simplicial complex on the vertex set $[n]=\{1,\dots,n\}$, with Stanley--Reisner ring
$\kk[\Gamma]$ over $\kk[x_1,\dots,x_n]$ and maximal graded ideal $\mm$. Then the $\Zn$-graded Hilbert
series of local cohomology is
\[
\sum_{a\in\Zn}\dim_\kk \Hloc{q}{\mm}(\kk[\Gamma])_a\,\mathbf x^a
=\sum_{\sigma\in\Gamma}\dim_\kk\rH^{\,q-|\sigma|-1}\!\big(\link_\Gamma(\sigma);\kk\big)
\prod_{j\in\sigma}\frac{x_j^{-1}}{1-x_j^{-1}} .
\]
In particular, for a multidegree $a$ with $\supp_-(a)=\{j:a_j<0\}=\sigma$ a face of $\Gamma$ and
$\supp_+(a)=\varnothing$,
\[
\dim_\kk \Hloc{q}{\mm}(\kk[\Gamma])_a=\dim_\kk\rH^{\,q-|\sigma|-1}\!\big(\link_\Gamma(\sigma);\kk\big),
\]
and all graded pieces not of this form (i.e.\ with $\supp_+(a)\neq\varnothing$, or with
$\supp_-(a)\notin\Gamma$) vanish.
\end{theorem}

\begin{proof}
The Hilbert-series identity is Hochster's formula; for the $\Zn$-graded, link-indexed form used here see
\cite[Theorem~2 and Lemma, p.~273]{Grabe}, \cite[Theorem~II.4.1]{StanleyCCA}, \cite[Theorem~5.3.8]{BrunsHerzog}, and
\cite[Theorem~13.13]{MillerSturmfels}. We record only the cochain-level identification that the persistent
version below requires. Fix $a\in\Zn$ and set $\sigma_-=\supp_-(a)$, $\sigma_+=\supp_+(a)$. By
Lemma~\ref{lem:localization} the degree-$p$ term of the multidegree-$a$ component of
$\check C^\bullet_\mm(\kk[\Gamma])$ has $\kk$-basis $\{F:\sigma_-\subseteq F,\ \sigma_+\cup F\in\Gamma,\ |F|=p\}$.
If $\sigma_+\neq\varnothing$ or $\sigma_-\notin\Gamma$ this component is acyclic (a cone, resp.\ void)
\cite[Lemma~5.3.5]{BrunsHerzog}, which gives the vanishing. If $\sigma_+=\varnothing$ and $\sigma:=\sigma_-\in\Gamma$, the nonzero basis elements in cohomological degree $p$ are indexed by the faces $F\in\Gamma$ satisfying $\sigma\subseteq F$ and $|F|=p$. Write $G:=F\setminus\sigma\in\link_\Gamma(\sigma)$. For every face $G$ disjoint from $\sigma$, define
\(N_\sigma(G):=\bigl|\{(s,g)\in\sigma\times G:s<g\}\bigr|\) and \(\varepsilon_\sigma(G):=(-1)^{N_\sigma(G)}\).
Let $e_F$ denote the basis element $x^a$ in the \v{C}ech summand indexed by $F$, and let $G^\vee$ denote the simplicial cochain dual to the oriented face $G$, with orientation induced by the natural ordering of $[n]$. Define
\[
\Psi_a^p(e_F)
:=
\varepsilon_\sigma(F\setminus\sigma)\,(F\setminus\sigma)^\vee.
\]
Since $|F\setminus\sigma|=p-|\sigma|$, these maps define an isomorphism of graded $\kk$-vector spaces
\[
\Psi_a^\bullet\colon
\check C^\bullet_\mm(\kk[\Gamma])_a
\longrightarrow
\widetilde C^{\,\bullet-|\sigma|-1}
\bigl(\link_\Gamma(\sigma);\kk\bigr).
\]

We verify that $\Psi_a^\bullet$ respects the differentials. Let $j\notin F$ with $F\cup\{j\}\in\Gamma$, and put $G=F\setminus\sigma$. Using the notation
\(
\nu_A(j):=\bigl|\{u\in A:u<j\}\bigr|
\)
for any $A\subseteq[n]$, we have $\nu_F(j)=\nu_G(j)+\nu_\sigma(j)$. Moreover,
\[
\varepsilon_\sigma(G\cup\{j\})
=
\varepsilon_\sigma(G)(-1)^{\nu_\sigma(j)}.
\]
Consequently,
\[
(-1)^{\nu_F(j)}\varepsilon_\sigma(G\cup\{j\})
=
(-1)^{\nu_G(j)}\varepsilon_\sigma(G).
\]
The left-hand side is the coefficient obtained by first applying the \v{C}ech differential and then $\Psi_a^{p+1}$, whereas the right-hand side is the coefficient obtained by first applying $\Psi_a^p$ and then the simplicial coboundary. Hence
\[
\Psi_a^{p+1}\circ d_{\check C}
=
\delta\circ\Psi_a^p.
\]
Thus $\Psi_a^\bullet$ is an isomorphism of cochain complexes, and therefore
\[
\Hloc{q}{\mm}(\kk[\Gamma])_a
\cong
\rH^{\,q-|\sigma|-1}
\bigl(\link_\Gamma(\sigma);\kk\bigr).
\]
This isomorphism is defined from the faces of $\Gamma$ alone; in particular it is natural for inclusions of
complexes on $[n]$, a fact used in the proof of Theorem~\ref{thm:persistent-hochster}.
\end{proof}

\begin{remark}[Relation to the standard references and applicability]\label{rem:hochster-source}
The formula in Theorem~\ref{thm:hochster-classical} is due to Hochster; the multigraded
identification of the graded pieces with reduced cohomology of links, with the explicit cochain
isomorphism used in the proof, is Gr\"abe's \cite[Theorem~2 and Lemma, p.~273]{Grabe}. It appears as
\cite[Theorem~II.4.1]{StanleyCCA} and \cite[Theorem~5.3.8]{BrunsHerzog}, and as
\cite[Theorem~13.13]{MillerSturmfels}. We have given the short proof in full because the precise
$\Zn$-graded, link-indexed form---rather than the coarsely graded statement---is what the persistent
theory below requires, and the same cochain-level isomorphism (now made natural in $\Gamma$) drives
the persistent formula of Theorem~\ref{thm:persistent-hochster}. Two points of comparison with the
textbook statements are worth making explicit, since the notation differs.

\emph{(a) Non-vertex variables.} We apply Theorem~\ref{thm:hochster-classical} to
$\Gamma=\link_\Delta(i)$ or $\del_\Delta(i)$ over $S'$, which need not use every variable as a vertex.
This is harmless: if $\{j\}\notin\Gamma$ then $x_j\in I_\Gamma$, so every \v{C}ech summand in a
multidegree with $j\in F$ vanishes, and the formula over $S'$ agrees with the one over the subring on
the vertices of $\Gamma$. The proof of Theorem~\ref{thm:hochster-classical} nowhere assumes $\Gamma$
uses every variable.

\emph{(b) Empty face and the degree-$(0,\dots,0)$ piece.} Since $\link_\Gamma(\varnothing)=\Gamma$, the
$\sigma=\varnothing$ term gives $\Hloc{q}{\mm}(\kk[\Gamma])_{(0,\dots,0)}\cong\widetilde
H^{q-1}(\Gamma;\kk)$. The degree convention is the standard reduced normalization, and two boundary
cases must be kept distinct. For the complex $\{\varnothing\}$ whose only face is the empty face,
$\widetilde H^{-1}(\{\varnothing\};\kk)=\kk$; for the void complex $\varnothing_{\mathrm{void}}$ with no
faces at all, $\widetilde H^{\bullet}(\varnothing_{\mathrm{void}};\kk)=0$. These occur for different
$\sigma$: when $\sigma$ is a \emph{facet} of $\Gamma$, the only $\tau$ with $\tau\cap\sigma=\varnothing$
and $\tau\cup\sigma\in\Gamma$ is $\tau=\varnothing$ (a nonempty such $\tau$ would give a face
$\tau\cup\sigma\supsetneq\sigma$, contradicting maximality), so $\link_\Gamma(\sigma)=\{\varnothing\}$
and $\widetilde H^{\,-1}(\link_\Gamma(\sigma);\kk)=\kk$; the void complex arises only when
$\sigma\notin\Gamma$, contributing $0$. Authors who phrase Hochster's formula with homology $\widetilde
H_{q-|\sigma|-1}$ rather than cohomology obtain the same dimensions, since over $\kk$ reduced homology
and cohomology of a finite complex are dual and equidimensional.
\end{remark}

Combining Theorems~\ref{thm:structure} and~\ref{thm:hochster-classical} yields the static
combinatorial description that we will persist. We state the piece that the filtration will use.

\begin{corollary}[Combinatorial description of vertex-prime local cohomology]\label{cor:static-comb}
Fix $i\in[n]$ and a multidegree $a=(c;b)$ with $b\in\ZZ^{[n]\setminus\{i\}}$. Let
$\sigma=\supp_-(b)\subseteq[n]\setminus\{i\}$ and assume $\supp_+(b)=\varnothing$ (all other
multidegrees give $0$). Then
\[
\dim_\kk\Hloc{q}{\pp_i}\!\big(\kk[\Delta]\big)_{(c;b)}=
\begin{cases}
\dim_\kk\rH^{\,q-|\sigma|-1}\!\big(\link_{\del_\Delta(i)}(\sigma);\kk\big), & c=0,\ \sigma\in\del_\Delta(i),\\[4pt]
\dim_\kk\rH^{\,q-|\sigma|-1}\!\big(\link_{\link_\Delta(i)}(\sigma);\kk\big), & c\geq1,\ \sigma\in\link_\Delta(i),\\[4pt]
0, & \text{otherwise.}
\end{cases}
\]
Moreover $\link_{\del_\Delta(i)}(\sigma)=\link_\Delta(\sigma)\cap\del_\Delta(i)$ and
$\link_{\link_\Delta(i)}(\sigma)=\link_\Delta(\sigma\cup\{i\})$.
\end{corollary}

\begin{proof}
Substitute Theorem~\ref{thm:hochster-classical} (for $\Gamma=\del_\Delta(i)$ when $c=0$, and
$\Gamma=\link_\Delta(i)$ when $c\geq1$) into Theorem~\ref{thm:structure}; the vanishing statements
match. For the last sentence: a face $\tau$ of $\del_\Delta(i)$ lies in $\link_{\del_\Delta(i)}(\sigma)$
iff $\tau\cap\sigma=\varnothing$, $i\notin\tau$, and $\tau\cup\sigma\in\Delta$, i.e.\
$\tau\in\link_\Delta(\sigma)$ and $i\notin\tau$, which is $\link_\Delta(\sigma)\cap\del_\Delta(i)$.
Similarly $\tau\in\link_{\link_\Delta(i)}(\sigma)$ iff $\tau\cap\sigma=\varnothing$ and
$\tau\cup\sigma\in\link_\Delta(i)$, i.e.\ $(\tau\cup\sigma)\cup\{i\}\in\Delta$, which (as
$i\notin\tau\cup\sigma$) means $\tau\in\link_\Delta(\sigma\cup\{i\})$.
\end{proof}

\section{The persistent theory}\label{sec:persistent}

\subsection{Filtrations and the reversed structure maps}\label{ssec:filtration}

Following \cite{SuwayyidWeiPSRT}, a monotonic function $f:\Delta\to\RR$ ($f(\tau)\le f(\sigma)$ whenever
$\tau\subseteq\sigma$) induces a filtration
\[
\Delta^t_f=\{\sigma\in\Delta:f(\sigma)\le t\},\qquad \Delta^s_f\subseteq\Delta^t_f\ (s\le t),
\]
a nested family of subcomplexes of $\Delta$ on the fixed vertex set $[n]$. We write $\Delta^t$ when
$f$ is understood. For $s\le t$ the inclusion $\Delta^s\subseteq\Delta^t$ gives the reverse inclusion
of Stanley--Reisner ideals $I_{\Delta^t}\subseteq I_{\Delta^s}$, hence a surjection of rings
\begin{equation}\label{eq:ring-surj}
\kk[\Delta^t]\twoheadrightarrow\kk[\Delta^s],\qquad s\le t .
\end{equation}

The functor $\Hloc{q}{\pp_i}=R^q\Gamma_{\pp_i}$ is covariant in the module argument. Applying it to
\eqref{eq:ring-surj} gives, for each $q$ and each $s\le t$, a $\Zn$-graded $\kk$-linear map
\begin{equation}\label{eq:struct-map}
\theta^{\,t\to s}_{i,q}:\ \Hloc{q}{\pp_i}\!\big(\kk[\Delta^t]\big)\longrightarrow
\Hloc{q}{\pp_i}\!\big(\kk[\Delta^s]\big),\qquad s\le t,
\end{equation}
running \emph{opposite} to the filtration direction. These maps are functorial:
$\theta^{\,s\to s}_{i,q}=\mathrm{id}$ and $\theta^{\,t\to s}_{i,q}\circ\theta^{\,u\to t}_{i,q}
=\theta^{\,u\to s}_{i,q}$ for $s\le t\le u$, because $\Gamma_{\pp_i}$ is a functor and
\eqref{eq:ring-surj} is functorial in the filtration. This is the local-cohomology counterpart of
the reversed Tor-side maps of \cite{SuwayyidWeiGraphs}; the persistence module it defines is an $\RR^{\mathrm{op}}$-indexed
(equivalently, a ``reversed-arrow'') persistence module.

\begin{remark}[Compatibility with the structure theorem]\label{rem:maps-split}
The surjection \eqref{eq:ring-surj} preserves the $x_i$-degree (it is a map of $\Zn$-graded rings).
Hence under Theorem~\ref{thm:structure} the map $\theta^{\,t\to s}_{i,q}$ respects the
deletion/link splitting: in $x_i$-degree $0$ it is the map
$\Hloc{q}{\mm'}(\kk[\del_{\Delta^t}(i)])\to\Hloc{q}{\mm'}(\kk[\del_{\Delta^s}(i)])$ induced by the
surjection $\kk[\del_{\Delta^t}(i)]\twoheadrightarrow\kk[\del_{\Delta^s}(i)]$ (which exists because
$\del_{\Delta^s}(i)\subseteq\del_{\Delta^t}(i)$), and in each $x_i$-degree $c\geq1$ it is the
analogous map for the link, induced by $\link_{\Delta^s}(i)\subseteq\link_{\Delta^t}(i)$. Both
inclusions of complexes hold because passing to deletion and link commutes with the inclusion of
subcomplexes on a fixed vertex set: $\sigma\in\Delta^s$ and ($\sigma\cup\{i\}\in\Delta^s$) imply the
same for $\Delta^t\supseteq\Delta^s$.
\end{remark}

We now extend the levelwise splitting of Remark~\ref{rem:maps-split} to an isomorphism of
persistence modules and record the remaining $x_i$-action. Recall that the variable $x_i$ is not
inverted in the \v{C}ech complex computing $\Hloc{\bullet}{\pp_i}$, so each
$\Hloc{q}{\pp_i}(\kk[\Delta^t])$ is a $\Zn$-graded module over $\kk[x_i]$, with $x_i$ raising the
$x_i$-degree. The maps $\theta^{t\to s}_{i,q}$ are $\kk[x_i]$-linear because they are induced by the
$\Zn$-graded ring surjections \eqref{eq:ring-surj}.

We first recall the persistence-module terminology that will be used below.
\begin{definition}[Persistence module, {\cite{ChazalDeSilvaGlisseOudot}}; functorial formulation
{\cite{BubenikScott}}]\label{def:pers-module}
Let $(P,\le)$ be a poset, regarded as a category with a unique morphism $s\to t$ whenever $s\le t$. A
\emph{$P$-indexed persistence module} of $\kk$-vector spaces is a functor $\mathbb U\colon
P\to\mathrm{Vec}_\kk$; concretely, a family $(U_t)_{t\in P}$ of $\kk$-vector spaces with transition
maps $u^s_t\colon U_s\to U_t$ for $s\le t$ satisfying $u^t_t=\mathrm{id}_{U_t}$ and
$u^t_r\circ u^s_t=u^s_r$ for $s\le t\le r$. A morphism of persistence modules is a natural
transformation of functors. We take $P=\RR$ with its usual order or its opposite $\RR^{\mathrm{op}}$;
the modules arising from \eqref{eq:struct-map}, whose maps run \emph{against} the filtration, are
$\RR^{\mathrm{op}}$-indexed. Such an $\mathbb U$ is \emph{pointwise finite-dimensional} \textup{(}pfd\textup{)}
if every $U_t$ is finite-dimensional.
\end{definition}

For $\alpha\in\Zn$, define the $\RR^{\mathrm{op}}$-indexed persistence module
\[
\mathbb V^q_{i,\alpha}(f)
:=\Big(\Hloc{q}{\pp_i}(\kk[\Delta^t_f])_\alpha,\ v^{t\to s}_\alpha\Big)_{t\in\RR},
\]
where $v^{t\to s}_\alpha$ is the restriction of $\theta^{t\to s}_{i,q}$ to the
multidegree-$\alpha$ component.

\begin{definition}[Deletion and link persistence modules]\label{def:DL-modules}
Fix a vertex $i$, a degree $q$, and a filtration $(\Delta^t_f)_{t\in\RR}$. For $b\in\ZZ^{[n]\setminus\{i\}}$
let $\mathbb D^q_{i,b}(f)$ and $\mathbb L^q_{i,b}(f)$ be the $\RR^{\mathrm{op}}$-indexed persistence
modules
\[
\mathbb D^q_{i,b}(f)\ :=\ \Big(\Hloc{q}{\mm'}(\kk[\del_{\Delta^t}(i)])_b,\ \eta^{t\to s}\Big),
\qquad
\mathbb L^q_{i,b}(f)\ :=\ \Big(\Hloc{q}{\mm'}(\kk[\link_{\Delta^t}(i)])_b,\ \zeta^{t\to s}\Big),
\]
where $\eta^{t\to s}$ and $\zeta^{t\to s}$ are the maps induced on $\Hloc{q}{\mm'}(-)_b$ by the ring
surjections $\kk[\del_{\Delta^t}(i)]\twoheadrightarrow\kk[\del_{\Delta^s}(i)]$ and
$\kk[\link_{\Delta^t}(i)]\twoheadrightarrow\kk[\link_{\Delta^s}(i)]$ of Remark~\ref{rem:maps-split}.
These satisfy the axioms of Definition~\ref{def:pers-module} by the functoriality already
established for $\theta^{t\to s}_{i,q}$ in \S\ref{ssec:filtration}, applied now to the nested
deletions $\del_{\Delta^s}(i)\subseteq\del_{\Delta^t}(i)$ and links
$\link_{\Delta^s}(i)\subseteq\link_{\Delta^t}(i)$ of Remark~\ref{rem:maps-split}.

\end{definition}

Proposition~\ref{prop:reduction} below shows that each of
$\mathbb D^q_{i,b}(f)$ and $\mathbb L^q_{i,b}(f)$ depends on $b$ only through
$\sigma=\supp_-(b)$, and that only finitely many squarefree representatives
$b=-\mathbf 1_\sigma$ are needed. We therefore write
$\mathbb D^q_{i,\sigma}(f)$ and $\mathbb L^q_{i,\sigma}(f)$ for the resulting
face-indexed modules.

\begin{theorem}[Summand decomposition of the persistence module]\label{thm:summand-iso}
Fix a vertex $i$ and a degree $q$, and let $(\Delta^t_f)_{t\in\RR}$ be the filtration induced by a
monotonic function $f\colon\Delta\to\RR$.
\begin{enumerate}[label=\textup{(\alph*)},nosep,leftmargin=2.2em]
\item \textup{(Multidegreewise.)} For every $c\in\ZZ$ and
$b\in\ZZ^{[n]\setminus\{i\}}$, there is an isomorphism of
$\RR^{\mathrm{op}}$-indexed persistence modules
\[
\mathbb V^q_{i,(c;b)}(f)\ \cong\
\begin{cases}
\mathbb D^q_{i,b}(f), & c=0,\\[2pt]
\mathbb L^q_{i,b}(f), & c\ge 1,\\[2pt]
0, & c\le -1 .
\end{cases}
\]
In particular, for fixed $b$, all modules $\mathbb V^q_{i,(c;b)}(f)$ with $c\ge1$ are canonically
isomorphic to $\mathbb L^q_{i,b}(f)$.

\item \textup{(Graded module-valued naturality.)} The levelwise isomorphisms of
Theorem~\ref{thm:structure} are natural in $t$. More precisely, let
\[
\Phi_t\colon \Hloc{q}{\pp_i}(\kk[\Delta^t_f])\xrightarrow{\ \cong\ }
\Hloc{q}{\mm'}(\kk[\del_{\Delta^t_f}(i)])\oplus
\Big(\Hloc{q}{\mm'}(\kk[\link_{\Delta^t_f}(i)])\otimes_\kk x_i\kk[x_i]\Big)
\]
be the isomorphism of $\Zn$-graded $S'$-modules from
Theorem~\ref{thm:structure}. Let $\eta^{t\to s}$ and $\zeta^{t\to s}$ denote the full
$\Zn$-graded $S'$-linear maps induced by the corresponding deletion and link surjections. Then,
for every $s\le t$,
\[
\Phi_s\circ\theta^{t\to s}_{i,q}
=
\bigl(\eta^{t\to s}\oplus
(\zeta^{t\to s}\otimes\mathrm{id}_{x_i\kk[x_i]})\bigr)\circ\Phi_t.
\]
Consequently, $\{\Phi_t\}_{t\in\RR}$ is a natural isomorphism of functors
\[
\RR^{\mathrm{op}}\longrightarrow \operatorname{GrMod}_{\Zn}(S').
\]
The displayed target is here regarded with its evident $S'$-module structure. Its
noncomponentwise $x_i$-action is described in Proposition~\ref{prop:xi-action}.
\end{enumerate}
\end{theorem}

\begin{proof}
\emph{(a)} Fix $c\in\ZZ$ and $b\in\ZZ^{[n]\setminus\{i\}}$. At each level $t$, the multidegree
form of Theorem~\ref{thm:structure} gives an isomorphism
\[
\phi_t^{c,b}\colon
\Hloc{q}{\pp_i}(\kk[\Delta^t_f])_{(c;b)}\xrightarrow{\ \cong\ }
\begin{cases}
\Hloc{q}{\mm'}(\kk[\del_{\Delta^t_f}(i)])_b, & c=0,\\[2pt]
\Hloc{q}{\mm'}(\kk[\link_{\Delta^t_f}(i)])_b, & c\ge1,\\[2pt]
0, & c\le-1 .
\end{cases}
\]
For $s\le t$, Remark~\ref{rem:maps-split} shows that the square formed by
$\phi_t^{c,b}$, $\phi_s^{c,b}$, and the corresponding transition maps commutes. In the case
$c=0$, the target transition map is $\eta^{t\to s}$; in the case $c\ge1$, it is
$\zeta^{t\to s}$; and in the case $c\le-1$, both sides are zero. Thus
$\{\phi_t^{c,b}\}_{t\in\RR}$ is a natural isomorphism of
$\RR^{\mathrm{op}}$-indexed persistence modules. For $c\ge1$, neither the levelwise
identification nor the transition maps depend on $c$, which gives the final assertion.

\emph{(b)} The map $\Phi_t$ is obtained by assembling the maps $\phi_t^{c,b}$ over all
multidegrees $(c;b)$. All maps in the asserted identity preserve the $\Zn$-grading. Therefore, the
identity may be checked in each multidegree separately, where it is exactly the commutative square
proved in part~(a). Each $\Phi_t$, $\eta^{t\to s}$, and $\zeta^{t\to s}$ is $S'$-linear, so the
result is a natural isomorphism with values in $\operatorname{GrMod}_{\Zn}(S')$.
\end{proof}

\begin{proposition}[The $S$-module structure of the decomposition]
\label{prop:xi-action}
For each $t$, let
\[
D_t^q
:=
\Hloc{q}{\mm'}\bigl(\kk[\del_{\Delta_f^t}(i)]\bigr),
\qquad
L_t^q
:=
\Hloc{q}{\mm'}\bigl(\kk[\link_{\Delta_f^t}(i)]\bigr).
\]
The inclusion
\[
\link_{\Delta_f^t}(i)\subseteq \del_{\Delta_f^t}(i)
\]
induces a quotient of Stanley--Reisner rings
\[
\kk[\del_{\Delta_f^t}(i)]
\twoheadrightarrow
\kk[\link_{\Delta_f^t}(i)]
\]
and hence a $\Zn$-graded $S'$-linear map
\[
\rho_t^q\colon D_t^q\longrightarrow L_t^q.
\]

Equip
\[
D_t^q\oplus\bigl(L_t^q\otimes_\kk x_i\kk[x_i]\bigr)
\]
with an $S=S'[x_i]$-module structure by retaining its natural
$S'$-module structure and defining multiplication by $x_i$ by
\[
x_i\cdot(u,0)
=
\bigl(0,\rho_t^q(u)\otimes x_i\bigr),
\qquad u\in D_t^q,
\]
and
\[
x_i\cdot\bigl(0,v\otimes x_i^c\bigr)
=
\bigl(0,v\otimes x_i^{c+1}\bigr),
\qquad
v\in L_t^q,\quad c\geq1.
\]
With this action, the isomorphism
\[
\Phi_t\colon
\Hloc{q}{\pp_i}\bigl(\kk[\Delta_f^t]\bigr)
\xrightarrow{\ \cong\ }
D_t^q\oplus\bigl(L_t^q\otimes_\kk x_i\kk[x_i]\bigr)
\]
of Theorem~\ref{thm:summand-iso} is an isomorphism of
$\Zn$-graded $S$-modules.

Moreover, for $s\leq t$, the maps $\rho_t^q$ are compatible with the
persistence structure:
\[
\rho_s^q\circ\eta^{t\to s}
=
\zeta^{t\to s}\circ\rho_t^q.
\]
Consequently, the family $\{\Phi_t\}_{t\in\RR}$ is a natural
isomorphism of $\Zn$-graded $S$-module-valued persistence modules.
\end{proposition}

\begin{proof}
Under the chain-level decomposition of
Lemma~\ref{lem:chain-split}, the $x_i$-degree-$0$ component of
$\check C^\bullet_T(\kk[\Delta_f^t])$ is identified with
\[
\check C^\bullet_{\mm'}
\bigl(\kk[\del_{\Delta_f^t}(i)]\bigr),
\]
whereas its $x_i$-degree-$1$ component is identified with
\[
x_i\check C^\bullet_{\mm'}
\bigl(\kk[\link_{\Delta_f^t}(i)]\bigr).
\]
Multiplication by $x_i$ sends a monomial in the deletion component to
the corresponding degree-$1$ monomial precisely when its support lies
in the link; otherwise the product is zero. Under the above
identifications, this is exactly the chain map induced by the quotient
\[
\kk[\del_{\Delta_f^t}(i)]
\twoheadrightarrow
\kk[\link_{\Delta_f^t}(i)].
\]
Passing to cohomology gives the map $\rho_t^q$. In every positive
$x_i$-degree, multiplication by $x_i$ is the ordinary shift
\[
v\otimes x_i^c\longmapsto v\otimes x_i^{c+1}.
\]
Therefore, $\Phi_t$ is $S$-linear.

Finally, the square of quotient maps
\[
\begin{tikzcd}
\kk[\del_{\Delta_f^t}(i)]
  \arrow[r,two heads]
  \arrow[d,two heads]
&
\kk[\link_{\Delta_f^t}(i)]
  \arrow[d,two heads]
\\
\kk[\del_{\Delta_f^s}(i)]
  \arrow[r,two heads]
&
\kk[\link_{\Delta_f^s}(i)]
\end{tikzcd}
\]
commutes. Applying $\Hloc{q}{\mm'}(-)$ yields
\[
\rho_s^q\circ\eta^{t\to s}
=
\zeta^{t\to s}\circ\rho_t^q,
\]
which proves naturality.
\end{proof}

\subsection{Persistent per-vertex local cohomology numbers}\label{ssec:numbers}

\begin{proposition}[Reduction to finitely many face-indexed modules]\label{prop:reduction}
Fix a vertex $i$, a degree $q$, and the filtration $(\Delta^t)_{t\in\RR}$ of a monotonic $f$. Let
$\alpha=(c;b)$ with $c\in\ZZ$, $b\in\ZZ^{[n]\setminus\{i\}}$, and set $\sigma:=\supp_-(b)$,
$\Gamma^t:=\del_{\Delta^t}(i)$ if $c=0$ and $\Gamma^t:=\link_{\Delta^t}(i)$ if $c\ge1$.
\begin{enumerate}[label=\textup{(\alph*)},nosep,leftmargin=2.2em]
\item \textup{(Support and finiteness.)} If $c\le-1$ or $\supp_+(b)\ne\varnothing$, then
$\mathbb V^q_{i,\alpha}=0$. Otherwise, for every $t$,
\[
\Hloc{q}{\pp_i}(\kk[\Delta^t])_\alpha\ \cong\ \Hloc{q}{\mm'}(\kk[\Gamma^t])_b\ \cong\
\rH^{\,q-|\sigma|-1}\big(\link_{\Gamma^t}(\sigma);\kk\big),
\]
a finite-dimensional $\kk$-vector space, nonzero only when $\sigma\in\Gamma^t$.
\item \textup{(Constancy.)} When $\supp_+(b)=\varnothing$, $\mathbb V^q_{i,\alpha}$ depends on $b$ only
through $\sigma$, and for $c\ge1$ is independent of $c$; it is therefore canonically isomorphic to the
module at the representative $b=-\mathbf 1_\sigma$, and, for $c\ge1$, at $c=1$.
\item \textup{(Finite family.)} Consequently the $\Zn$-graded family
$\{\mathbb V^q_{i,\alpha}\}_{\alpha\in\Zn}$ reduces, up to canonical isomorphism, to the finite
face-indexed collection
\[
\{\mathbb V^q_{i,(0;-\mathbf 1_\sigma)}:\sigma\in\del_\Delta(i)\}\ \cup\
\{\mathbb V^q_{i,(1;-\mathbf 1_\sigma)}:\sigma\in\link_\Delta(i)\},
\]
each pointwise finite-dimensional. The total sum $\bigoplus_{\alpha\in\Zn}\mathbb V^q_{i,\alpha}$ is
\emph{not} pointwise finite-dimensional whenever a link term is nonzero, since it is repeated in every
$x_i$-degree $c\geq1$, or whenever
$\rH^{\,q-|\sigma|-1}(\link_{\Gamma^t}(\sigma);\kk)\ne0$
for some $\sigma\ne\varnothing$, being then nonzero on the entire orthant
$\{b:\supp_+(b)=\varnothing,\ \supp_-(b)=\sigma\}$; the invariant is the finite family, not that sum.
\end{enumerate}
\end{proposition}

\begin{proof}
\emph{(a)} The two isomorphisms are the multidegree form of Theorem~\ref{thm:structure} and
Theorem~\ref{thm:hochster-classical}; the target is a reduced cohomology group of the finite complex
$\link_{\Gamma^t}(\sigma)$, hence finite-dimensional, and it vanishes when $\sigma\notin\Gamma^t$
(void link), when $\supp_+(b)\ne\varnothing$ (Theorem~\ref{thm:hochster-classical}), and when $c\le-1$
(Theorem~\ref{thm:structure}). 

\emph{(b)} Suppose $\supp_+(b)=\varnothing$. By the cochain isomorphism
constructed in the proof of Theorem~\ref{thm:hochster-classical}, the
multidegree-$b$ component $\check C^\bullet_{\mm'}(\kk[\Gamma^t])_b$
is naturally isomorphic to
$\widetilde C^{\,\bullet-|\sigma|-1}
\bigl(\link_{\Gamma^t}(\sigma);\kk\bigr)$, where
$\sigma=\supp_-(b)$.
Under this identification, for every $s\leq t$, the cochain map induced
by the surjection $\kk[\Gamma^t]\twoheadrightarrow\kk[\Gamma^s]$
corresponds to the restriction cochain map induced by
$\link_{\Gamma^s}(\sigma)\subseteq\link_{\Gamma^t}(\sigma)$. This
remains valid when $\sigma\notin\Gamma^s$, in which case both target
complexes are zero. Hence the persistence module depends on $b$ only
through $\sigma=\supp_-(b)$. For $c\geq1$, the surjection
\eqref{eq:ring-surj} acts by the same rule in every $x_i$-degree by
Theorem~\ref{thm:summand-iso}(a), so $\mathbb V^q_{i,\alpha}$ is
independent of $c$ in that range. Canonically isomorphic persistence
modules have equal barcodes. 
\emph{(c)} is immediate from (b), the faces of $\del_\Delta(i)$ and
$\link_\Delta(i)$ being finite in number; the non-pfd statement follows
because every nonzero link term is repeated in all $x_i$-degrees
$c\geq1$, while for $\sigma\ne\varnothing$ a nonzero term is repeated
over the entire orthant
$\{b:\supp_+(b)=\varnothing,\ \supp_-(b)=\sigma\}$.
\end{proof}

\begin{definition}[Persistent per-vertex local cohomology numbers]\label{def:numbers}
Let $(\Delta^t)_{t\in\RR}$ be the filtration induced by a monotonic
$f\colon\Delta\to\RR$. For a vertex $i$, a cohomological degree $q$, a
multidegree $\alpha\in\Zn$, and $s\le t$, define
\[
\lambda^{\,q,\,s\to t}_{i,\alpha}
:=\rank\Big(
\theta^{\,t\to s}_{i,q}\colon
\Hloc{q}{\pp_i}(\kk[\Delta^t])_\alpha
\longrightarrow
\Hloc{q}{\pp_i}(\kk[\Delta^s])_\alpha
\Big).
\]
This rank is finite by Proposition~\ref{prop:reduction}(a). If
$\alpha=(c;b)$ and $\sigma=\supp_-(b)$, Proposition~\ref{prop:reduction}
shows that the value vanishes unless $\supp_+(b)=\varnothing$ and $\sigma$ is
a face of $\del_\Delta(i)$ when $c=0$, or of $\link_\Delta(i)$ when $c\ge1$.
In the nonzero cases it depends only on $\sigma$ and on whether $c=0$ or
$c\ge1$. Using the representatives $b=-\mathbf 1_\sigma$ and $c\in\{0,1\}$,
write
\[
\lambda^{\,q,\,s\to t}_{i,\del,\sigma}
:=\lambda^{\,q,\,s\to t}_{i,(0;-\mathbf 1_\sigma)}
\quad(\sigma\in\del_\Delta(i)),
\qquad
\lambda^{\,q,\,s\to t}_{i,\link,\sigma}
:=\lambda^{\,q,\,s\to t}_{i,(1;-\mathbf 1_\sigma)}
\quad(\sigma\in\link_\Delta(i)).
\]
We do not sum over all multidegrees: by Proposition~\ref{prop:reduction}(c),
the resulting total module need not be pointwise finite-dimensional. When $s=t$, these numbers recover \(\dim_\kk\Hloc{q}{\mm'}\bigl(\kk[\del_{\Delta^t}(i)]\bigr)_{-\mathbf 1_\sigma}\) and \(\dim_\kk\Hloc{q}{\mm'}\bigl(\kk[\link_{\Delta^t}(i)]\bigr)_{-\mathbf 1_\sigma}\), respectively.
\end{definition}

\noindent The superscript $s\to t$ records the filtration direction
$\Delta^s\subseteq\Delta^t$, whereas the underlying map
$\theta^{t\to s}_{i,q}$ runs from level $t$ to level $s$. This agrees with
the reversed-arrow convention of \eqref{eq:struct-map} and with the Tor-side
convention of \cite{SuwayyidWeiGraphs}.

\subsection{The persistent links--Hochster formula}\label{ssec:persistent-hochster}

We now identify the persistent numbers with ranks of maps on reduced (co)homology of links and
deletions of \emph{induced subcomplexes of the filtration}. This is the local-cohomology, vertex-
local analogue of the persistent Hochster formula of \cite{SuwayyidWeiGraphs}.

\begin{theorem}[Persistent links--Hochster formula]\label{thm:persistent-hochster}
Let $(\Delta^t)_{t\in\RR}$ be the filtration of a monotonic $f:\Delta\to\RR$, fix $i\in[n]$,
$q\geq0$, $s\le t$, and a multidegree $\alpha=(c;b)$ with $b\in\ZZ^{[n]\setminus\{i\}}$. Write
$\sigma=\supp_-(b)$ and assume $\supp_+(b)=\varnothing$ (all other $\alpha$ give
$\lambda^{q,s\to t}_{i,\alpha}=0$). Set $d=q-|\sigma|-1$. Then:
\[
\lambda^{\,q,\,s\to t}_{i,\alpha}=
\begin{cases}
\rank\Big(\rH^{\,d}\big(\link_{\Delta^t}(\sigma)\cap\del_{\Delta^t}(i);\kk\big)\to
\rH^{\,d}\big(\link_{\Delta^s}(\sigma)\cap\del_{\Delta^s}(i);\kk\big)\Big), & c=0,\\[6pt]
\rank\Big(\rH^{\,d}\big(\link_{\Delta^t}(\sigma\cup\{i\});\kk\big)\to
\rH^{\,d}\big(\link_{\Delta^s}(\sigma\cup\{i\});\kk\big)\Big), & c\geq1,
\end{cases}
\]
where each map is the restriction in reduced cohomology induced by the inclusion of the
corresponding subcomplex of $\Delta^s$ into that of $\Delta^t$ (with the convention $\link_\Gamma(\sigma)=\varnothing_{\mathrm{void}}$ when $\sigma\notin\Gamma$, so
a summand's reduced cohomology is $0$ at any level where its indexing face is not yet present). Equivalently, by duality over $\kk$,
each rank equals the rank of the induced map on reduced \emph{homology} in degree $d$, running in the
filtration direction $\Delta^s\hookrightarrow\Delta^t$.
\end{theorem}

\begin{proof}
The proof has two steps: extending the chain-level identification of Theorem~\ref{thm:hochster-classical}
to a \emph{natural} statement in the complex, then applying it to the deletion and link filtrations.

\emph{Step 1: naturality.} By Theorem~\ref{thm:hochster-classical}, for any complex $\Gamma$ on $[n]$ and any multidegree $b$
with $\supp_+(b)=\varnothing$ and $\sigma:=\supp_-(b)\in\Gamma$ there is an isomorphism of
finite-dimensional $\kk$-vector spaces
\begin{equation}\label{eq:grabe}
\Hloc{q}{\mm}(\kk[\Gamma])_b\ \cong\ \rH^{\,q-|\sigma|-1}\big(\link_\Gamma(\sigma);\kk\big).
\end{equation}
We recall the cochain-level construction from the proof of that theorem, since it is its
\emph{naturality} that Step~2 requires. In cohomological degree $p$ the multidegree-$b$ component of
$\check C^\bullet_\mm(\kk[\Gamma])$ is $\bigoplus_{|F|=p}\big(\kk[\Gamma]_{x_F}\big)_b$, where $F$ ranges
over the $p$-element subsets of $[n]$ indexing the \v{C}ech summands $\kk[\Gamma]_{x_F}$ of
\eqref{eq:cech}, with $x_F=\prod_{j\in F}x_j$. By Lemma~\ref{lem:localization}, for this fixed $b$ each
such summand is at most one-dimensional---it is $\kk\cdot x^b$ when $\sigma\subseteq F$ and $F\in\Gamma$,
and $0$ otherwise---so the nonzero basis elements in multidegree $b$ are indexed by the finitely many
faces $F$ of $\Gamma$ containing $\sigma$, and both sides of \eqref{eq:grabe} are finite-dimensional
$\kk$-vector spaces. The correspondence
\[
\{F\in\Gamma:\sigma\subseteq F\}\ \xrightarrow{\ \sim\ }\ \link_\Gamma(\sigma),
\qquad
F\mapsto F\setminus\sigma,
\]
together with the signs introduced in the proof of
Theorem~\ref{thm:hochster-classical}, defines the cochain isomorphism
\[
\Psi_{\Gamma,b}^\bullet\colon
\check C^\bullet_\mm(\kk[\Gamma])_b
\longrightarrow
\widetilde C^{\,\bullet-|\sigma|-1}
\bigl(\link_\Gamma(\sigma);\kk\bigr),
\]
under which the basis element $e_F$ is sent to
$\varepsilon_\sigma(F\setminus\sigma)(F\setminus\sigma)^\vee$. Since
$|F\setminus\sigma|=|F|-|\sigma|$, \eqref{eq:grabe} follows on passing to cohomology.

This isomorphism is natural for inclusions of complexes on $[n]$: for a subcomplex
$\Gamma'\subseteq\Gamma$, the surjection $\kk[\Gamma]\twoheadrightarrow\kk[\Gamma']$ induces on the
multidegree-$b$ components the map carrying the basis monomial indexed by $F$ to itself if
$F\in\Gamma'$ and to $0$ otherwise. Under $\Psi_{\Gamma,b}^\bullet$ and
$\Psi_{\Gamma',b}^\bullet$, this is exactly the restriction cochain map induced by
$\link_{\Gamma'}(\sigma)\subseteq\link_\Gamma(\sigma)$, because the factor
$\varepsilon_\sigma(F\setminus\sigma)$ is independent of the ambient complex.
Passing to cohomology, the square
\[
\begin{tikzcd}
\Hloc{q}{\mm}(\kk[\Gamma])_b \arrow[r,"\cong"] \arrow[d]
&
\rH^{\,q-|\sigma|-1}(\link_\Gamma(\sigma);\kk)
\arrow[d,"\mathrm{restr}"]
\\
\Hloc{q}{\mm}(\kk[\Gamma'])_b \arrow[r,"\cong"]
&
\rH^{\,q-|\sigma|-1}(\link_{\Gamma'}(\sigma);\kk)
\end{tikzcd}
\]
commutes, the left vertical map being induced by
$\kk[\Gamma]\twoheadrightarrow\kk[\Gamma']$ and the right by the link inclusion. When
$\sigma\notin\Gamma'$, the lower row is $0$: by
Theorem~\ref{thm:hochster-classical},
$\Hloc{q}{\mm}(\kk[\Gamma'])_b=0$, while
$\link_{\Gamma'}(\sigma)=\varnothing_{\mathrm{void}}$ gives
$\rH^{\,q-|\sigma|-1}(\link_{\Gamma'}(\sigma);\kk)=0$. Thus the square commutes for every subcomplex
$\Gamma'\subseteq\Gamma$ on $[n]$, whether or not $\sigma\in\Gamma'$; if $\sigma\notin\Gamma$, then
$\sigma\notin\Gamma'$ as well, both ends vanish, and
$\lambda^{q,s\to t}_{i,\alpha}=0$.

\emph{Step 2: application to deletion and link of $i$.} By Theorem~\ref{thm:structure} in the form of
its multidegree statement, and Remark~\ref{rem:maps-split}, the structure map $\theta^{t\to s}_{i,q}$
in multidegree $\alpha=(c;b)$ is: for $c=0$, the map
$\Hloc{q}{\mm'}(\kk[\del_{\Delta^t}(i)])_b\to\Hloc{q}{\mm'}(\kk[\del_{\Delta^s}(i)])_b$ induced by the
inclusion $\del_{\Delta^s}(i)\subseteq\del_{\Delta^t}(i)$;
and for $c\ge1$, the analogous map for $\link_{\Delta^s}(i)\subseteq\link_{\Delta^t}(i)$. Applying the
commuting square of Step~1 with $\Gamma=\del_{\Delta^t}(i)\supseteq\Gamma'=\del_{\Delta^s}(i)$ (resp.\
$\Gamma=\link_{\Delta^t}(i)\supseteq\Gamma'=\link_{\Delta^s}(i)$) and taking ranks, we obtain that
$\lambda^{q,s\to t}_{i,\alpha}=\rank(\theta^{t\to s}_{i,q}|_\alpha)$ equals the rank of the restriction
map on $\widetilde H^{\,q-|\sigma|-1}$ of the corresponding links. Finally, the identifications
\[
\link_{\del_{\Delta'}(i)}(\sigma)=\link_{\Delta'}(\sigma)\cap\del_{\Delta'}(i),
\qquad
\link_{\link_{\Delta'}(i)}(\sigma)=\link_{\Delta'}(\sigma\cup\{i\})
\]
(Corollary~\ref{cor:static-comb}, applied at each filtration level $\Delta'=\Delta^s,\Delta^t$) turn
these into the two displayed formulas. All other multidegrees $\alpha$ (those with
$\supp_+(b)\ne\varnothing$, or with $c\le-1$) give $\Hloc{q}{\pp_i}(\kk[\Delta^\bullet])_\alpha=0$ on
both ends by Theorem~\ref{thm:structure}, hence $\lambda^{q,s\to t}_{i,\alpha}=0$.

For the homological reformulation, recall that over a field $\kk$ the reduced simplicial cochain
complex is the degreewise $\kk$-linear dual of the reduced chain complex, with coboundary the transpose
of boundary; since $Hom_\kk(-,\kk)$ is exact over a field, taking cohomology commutes with
dualization, so $\rH^{\,d}(X;\kk)\cong\big(\rH_{\,d}(X;\kk)\big)^{\ast}$ naturally in $X$ (universal
coefficients over a field). Consequently the restriction map in reduced cohomology induced by a
subcomplex inclusion $\link_{\Gamma^s}(\sigma)\hookrightarrow\link_{\Gamma^t}(\sigma)$ is the
$\kk$-dual of the map in reduced homology induced by the \emph{same} inclusion, the latter running in
the filtration direction $s\to t$. As a linear map of finite-dimensional $\kk$-vector spaces and its
dual have equal rank, the two ranks coincide, which is the homological form of the theorem.
\end{proof}

\begin{remark}[Relation to the Tor-side persistent Hochster formula]\label{rem:tor-side}
The homological reformulation makes the comparison with \cite{SuwayyidWeiGraphs} obvious. Their
persistent Hochster formula expresses Tor-side persistent Betti numbers of $\kk[\Delta^\bullet]$ as ranks of
induced maps on $\rH^{|W|-i-1}((\Delta^\bullet)_W;\kk)$ over induced subcomplexes $W$. Here the
relevant complexes are not induced subcomplexes on a vertex set $W$, but \emph{links and deletions of
a single vertex} $i$ (and their iterated links along $\sigma$). The vertex-prime construction thus
records the filtration's reduced (co)homology localized at the star of $i$, rather than summed over
all $W$. This is the precise sense in which it is a \emph{per-vertex} refinement.
\end{remark}

\subsection{Finiteness and interval decomposition}\label{ssec:finite-type}

We now organize the persistent numbers into persistence modules and show they are interval
decomposable. Fix $i$, $q$, and a multidegree $\alpha$, and consider the family defined at \ref{def:pers-module}
\[
\mathbb V^{q}_{i,\alpha}\ :=\ \Big(\,\Hloc{q}{\pp_i}(\kk[\Delta^t])_\alpha,\ \theta^{t\to s}_{i,q}\,\Big),
\]
indexed by $t\in\RR$ with the reversed structure maps
\eqref{eq:struct-map}. Since $\RR^{\mathrm{op}}$ is totally ordered,
the standard interval-decomposition theorem applies directly. We write
$\kk_J$ for the $\RR^{\mathrm{op}}$-indexed interval module supported
on an interval $J\subseteq\RR$. We record the relevant finiteness.

\begin{theorem}[Finite type and interval decomposition]\label{thm:finite-type}
For every vertex $i$, degree $q$, and multidegree $\alpha\in\Zn$, the
persistence module $\mathbb V^q_{i,\alpha}(f)$ is pointwise
finite-dimensional and constructible with critical values contained in the
finite set
\[
\operatorname{Crit}(f):=\{f(\tau):\tau\in\Delta\}.
\]
Equivalently, as a finite-type $\RR^{\mathrm{op}}$-indexed persistence
module, it has a finite interval decomposition
\[
\mathbb V^q_{i,\alpha}(f)\cong
\bigoplus_{\ell=1}^{N_{i,q,\alpha}}\kk_{J_\ell},
\]
unique up to permutation of the interval summands. The finite family of
isomorphism types that must be retained is precisely the face-indexed family
listed in Proposition~\ref{prop:reduction}(c); in particular, every such
module is $q$-tame.
\end{theorem}

\begin{proof}
Pointwise finite-dimensionality follows from
Proposition~\ref{prop:reduction}(a). Since $\Delta$ is finite,
$\operatorname{Crit}(f)$ is finite. If an interval $[s,t]$ contains no value
of $\operatorname{Crit}(f)$, then $\Delta^s_f=\Delta^t_f$, and the induced
structure map on every multidegree is the identity. Thus $\mathbb V^q_{i,\alpha}(f)$ is a constructible finite-type
$\RR^{\mathrm{op}}$-indexed persistence module. A pointwise finite-dimensional finite-type persistence
module admits a finite interval decomposition; uniqueness follows from the
standard uniqueness theorem for interval decompositions
\cite[Theorem~1.1]{CrawleyBoevey} (see also
\cite[Theorem~1.1]{BotnanCrawleyBoevey}). The final assertion follows from
Proposition~\ref{prop:reduction}(c), and $q$-tameness follows from pointwise
finite-dimensionality.
\end{proof}

\begin{definition}[Per-vertex local-cohomology barcodes]\label{def:barcode}
For a vertex $i$, degree $q$, and multidegree $\alpha$, let
$\mathbb V^q_{i,\alpha}\cong\bigoplus_{\ell}\kk_{J_\ell}$ be the interval
decomposition supplied by Theorem~\ref{thm:finite-type}, unique by that theorem.
The \emph{per-vertex local-cohomology barcode} in multidegree $\alpha$ is the
multiset of its intervals,
\[
\mathcal B^q_{i,\alpha}(f)\ :=\ \{\,J_\ell\,\}_{\ell},
\qquad\text{where }\ \mathbb V^q_{i,\alpha}\cong\bigoplus_{\ell}\kk_{J_\ell}.
\]
By Proposition~\ref{prop:reduction}(b), $\mathbb V^q_{i,\alpha}$ depends on
$\alpha=(c;b)$ only through the case $c\in\{0\}$ vs.\ $c\ge1$ and the face
$\sigma=\supp_-(b)$, so $\mathcal B^q_{i,\alpha}(f)$ is determined by the squarefree
representatives $\alpha=(0;-\mathbf 1_\sigma)$ and $\alpha=(1;-\mathbf 1_\sigma)$;
we abbreviate the resulting face-indexed barcodes by
\[
\mathcal B^q_{i,\del,\sigma}(f)\ :=\ \mathcal B^q_{i,(0;-\mathbf 1_\sigma)}(f)
\quad(\sigma\in\del_\Delta(i)),
\qquad
\mathcal B^q_{i,\link,\sigma}(f)\ :=\ \mathcal B^q_{i,(1;-\mathbf 1_\sigma)}(f)
\quad(\sigma\in\link_\Delta(i)).
\]
The \emph{per-vertex barcode of $(i,q)$} is the finite face-indexed collection
\(\mathcal B_i^q(f):=\bigl(\{\mathcal B^q_{i,\del,\sigma}(f)\}_{\sigma},\)
\newline
\(\{\mathcal B^q_{i,\link,\sigma}(f)\}_{\sigma}\bigr)\).
Equivalently, by
Theorem~\ref{thm:persistent-hochster}, $\mathcal B^q_{i,\del,\sigma}(f)$ and
$\mathcal B^q_{i,\link,\sigma}(f)$ are the barcodes of the reversed one-parameter
persistence modules with ranks $\lambda^{q,s\to t}_{i,\del,\sigma}$ and
$\lambda^{q,s\to t}_{i,\link,\sigma}$, i.e.\ of the reduced cohomology of the links
and deletions of the induced filtration near vertex $i$. 
\end{definition}

The persistent links--Hochster formula makes the per-vertex barcodes explicit in the
two lowest reduced cohomological degrees.  We record these formulas for the
face-indexed barcodes of Definition~\ref{def:barcode}.

\begin{proposition}
\label{cor:minus-one-bars}
Fix a vertex $i$ and a face
$\sigma\subseteq[n]\setminus\{i\}$, and set $q=|\sigma|$, so that
$q-|\sigma|-1=-1$.
\begin{enumerate}[label=\textup{(\alph*)},leftmargin=2.2em]
	\item If $\sigma\in\del_\Delta(i)$, then the deletion barcode
	$\mathcal B^q_{i,\del,\sigma}(f)$ consists of the single possible interval
	\begin{equation}
	\label{eq:minus-one-deletion}
		\left[
		f(\sigma),
		\min_{\substack{v\notin\sigma\cup\{i\}\\
		\sigma\cup\{v\}\in\Delta}}
		f(\sigma\cup\{v\})
		\right).
	\end{equation}

	\item If $\sigma\in\link_\Delta(i)$, equivalently
	$\tau:=\sigma\cup\{i\}\in\Delta$, then the link barcode
	$\mathcal B^q_{i,\link,\sigma}(f)$ consists of the single possible interval
	\begin{equation}
	\label{eq:minus-one-link}
		\left[
		f(\tau),
		\min_{\substack{v\notin\tau\\
		\tau\cup\{v\}\in\Delta}}
		f(\tau\cup\{v\})
		\right).
	\end{equation}
\end{enumerate}
In both formulas the minimum of the empty set is understood to be $\infty$,
and a zero-length interval is omitted.
\end{proposition}

\begin{proof}
By Theorem~\ref{thm:persistent-hochster}, the deletion barcode is the barcode of
the reversed reduced-cohomology persistence module
\[
\widetilde H^{-1}
\bigl(
\link_{\Delta^t}(\sigma)\cap\del_{\Delta^t}(i);\kk
\bigr),
\]
while the link barcode is the barcode of
\[
\widetilde H^{-1}
\bigl(
\link_{\Delta^t}(\sigma\cup\{i\});\kk
\bigr).
\]
For a simplicial complex $K$, one has
$\widetilde H^{-1}(K;\kk)\cong\kk$ precisely when
$K=\{\varnothing\}$, and it is zero for both the void complex and every
complex containing a vertex.

For the deletion summand, the empty face enters
$\link_{\Delta^t}(\sigma)\cap\del_{\Delta^t}(i)$ exactly when
$\sigma$ enters the filtration, namely at $t=f(\sigma)$.  The complex remains
equal to $\{\varnothing\}$ until the first vertex $v\neq i$ satisfying
$\sigma\cup\{v\}\in\Delta$ appears in the link, which occurs at
$t=f(\sigma\cup\{v\})$.  This gives \eqref{eq:minus-one-deletion}.

For the link summand, writing $\tau=\sigma\cup\{i\}$, the empty face enters
$\link_{\Delta^t}(\tau)$ at $t=f(\tau)$ and survives until the first vertex
$v\notin\tau$ with $\tau\cup\{v\}\in\Delta$ appears.  This gives
\eqref{eq:minus-one-link}.  Since $f$ is monotonic, no higher-dimensional
coface can enter before all of its one-vertex subcofaces.
\end{proof}

\begin{proposition}
\label{cor:msf-bars}
Fix a vertex $i$ and a face
$\sigma\subseteq[n]\setminus\{i\}$, and set $q=|\sigma|+1$, so that
$q-|\sigma|-1=0$.  Let $(K^t)_{t\in\RR}$ denote either of the two filtrations
\[
K^t_{\del}
=
\link_{\Delta^t}(\sigma)\cap\del_{\Delta^t}(i),
\qquad
K^t_{\link}
=
\link_{\Delta^t}(\sigma\cup\{i\}),
\]
corresponding respectively to
$\mathcal B^q_{i,\del,\sigma}(f)$ and
$\mathcal B^q_{i,\link,\sigma}(f)$.

Assume that every vertex of the final $1$-skeleton
$G=(V,E)$ of $K^\bullet$ appears at the same filtration value $b$.
For each edge $e\in E$, let $w(e)$ be its filtration value, and let $F$ be any
minimum spanning forest of the edge-weighted graph $(G,w)$.  If $c$ is the
number of connected components of $G$, then the corresponding per-vertex
barcode is
\begin{equation}
\label{eq:msf-barcode}
	\bigl\{[b,w(e)):e\in F\bigr\}
	\ \cup\
	\bigl\{[b,\infty)^{\,c-1}\bigr\},
\end{equation}
with zero-length intervals omitted.
\end{proposition}

\begin{proof}
By Theorem~\ref{thm:persistent-hochster}, the relevant per-vertex barcode is the
barcode of the reversed persistence module
$\widetilde H^0(K^t;\kk)$.  Since $H^0$ depends only on the $1$-skeleton,
higher-dimensional simplices play no role.

At the common vertex-birth value $b$, the reduced cohomology has dimension
$|V|-1$.  Process the edges in nondecreasing order of $w(e)$, as in Kruskal's
algorithm.  Whenever an edge joins two previously distinct connected
components, the number of components decreases by one and one independent
class in $\widetilde H^0$ disappears at that edge weight.  The edges producing
such mergers are precisely the edges of a minimum spanning forest $F$.
Edges whose endpoints are already connected do not change
$\widetilde H^0$.

Since $F$ has $|V|-c$ edges, these give the finite intervals
$[b,w(e))$, $e\in F$.  The final graph has $c$ connected components, so
$\dim_\kk\widetilde H^0(G;\kk)=c-1$, giving the remaining $c-1$ essential
intervals $[b,\infty)$.  This proves \eqref{eq:msf-barcode}.
\end{proof}

\subsection{Stability}\label{ssec:stability}

We prove a stability theorem for the per-vertex barcodes. We use the interleaving
formalism for persistence modules and the isometry theorem of Chazal--de Silva--Glisse--Oudot, as in
Suwayyid--Wei. We first record how monotone functions close in sup-norm induce interleaved
filtrations, then transport the interleaving through $\Hloc{q}{\pp_i}$.

Throughout, $f,g:\Delta\to\RR$ are monotonic and $\delta:=\|f-g\|_\infty=\sup_{\sigma\in\Delta}
|f(\sigma)-g(\sigma)|$.

\begin{lemma}[Interleaved filtrations]\label{lem:interleave-filt}
For all $t\in\RR$,
\[
\Delta^t_f\subseteq\Delta^{t+\delta}_g\subseteq\Delta^{t+2\delta}_f .
\]
Consequently, applying the deletion (resp.\ link) of a fixed vertex $i$, which preserves inclusions
of complexes on $[n]$,
\[
\del_{\Delta^t_f}(i)\subseteq\del_{\Delta^{t+\delta}_g}(i)\subseteq\del_{\Delta^{t+2\delta}_f}(i),
\qquad
\link_{\Delta^t_f}(i)\subseteq\link_{\Delta^{t+\delta}_g}(i)\subseteq\link_{\Delta^{t+2\delta}_f}(i).
\]
\end{lemma}

\begin{proof}
If $\sigma\in\Delta^t_f$ then $f(\sigma)\le t$, so $g(\sigma)\le f(\sigma)+\delta\le t+\delta$, i.e.\
$\sigma\in\Delta^{t+\delta}_g$; the second inclusion is symmetric. Deletion is the induced
subcomplex on $[n]\setminus\{i\}$ and link is $\{\tau\not\ni i:\tau\cup\{i\}\in\Delta'\}$; both are
monotone in $\Delta'$ under inclusion of complexes on the fixed vertex set, giving the displayed
chains.
\end{proof}

We recall the two external ingredients we use.

\begin{definition}[Interleaving of persistence modules {\cite[\S3]{ChazalDeSilvaGlisseOudot}}]\label{def:interleaving}
Let $\mathbb U=(U_t,u^s_t)$ and $\mathbb W=(W_t,w^s_t)$ be $\RR$-indexed persistence modules and
$\delta\ge0$. A \emph{$\delta$-interleaving} is a pair of families of linear maps
$\varphi_t:U_t\to W_{t+\delta}$ and $\psi_t:W_t\to U_{t+\delta}$ ($t\in\RR$) such that:
(a) $\varphi$ and $\psi$ are morphisms of persistence modules of degree $\delta$, i.e.\
$w^{s+\delta}_{t+\delta}\varphi_s=\varphi_t u^s_t$ and $u^{s+\delta}_{t+\delta}\psi_s=\psi_t w^s_t$
for all $s\le t$; and (b) the two composites equal the internal transition maps over $2\delta$:
$\psi_{t+\delta}\varphi_t=u^t_{t+2\delta}$ and $\varphi_{t+\delta}\psi_t=w^t_{t+2\delta}$ for all $t$.
The \emph{interleaving distance} is $d_i(\mathbb U,\mathbb W)=\inf\{\delta\ge0:\mathbb U,\mathbb W\text{ are }\delta\text{-interleaved}\}$.
\end{definition}

\begin{theorem}[Isometry/stability theorem {\cite[Thms.~4.9, 4.11, 5.14]{ChazalDeSilvaGlisseOudot}}]\label{thm:cdsgo}
Let $\mathbb U,\mathbb W$ be $q$-tame $\RR$-indexed persistence modules. If $\mathbb U$ and $\mathbb W$
are $\delta$-interleaved, then there is a $\delta$-matching between their persistence diagrams
$\mathrm{dgm}(\mathbb U)$ and $\mathrm{dgm}(\mathbb W)$; consequently
$d_b(\mathrm{dgm}(\mathbb U),\mathrm{dgm}(\mathbb W))\le d_i(\mathbb U,\mathbb W)$.
\textup{(}$q$-tameness guarantees the diagrams exist \cite[\S2.8, \S3]{ChazalDeSilvaGlisseOudot}.\textup{)}
\end{theorem}

\begin{definition}[Metric on the face-indexed collection]\label{def:coll-metric}
Let $\mathcal I_i=\del_\Delta(i)\sqcup\link_\Delta(i)$ be the fixed finite index set of
Definition~\ref{def:barcode}, and for $\tau\in\mathcal I_i$ write $\mathcal B^q_{i,\tau}$ for the
corresponding component of $\mathcal B^q_i$ \textup{(}i.e.\ $\mathcal B^q_{i,\del,\sigma}$ or
$\mathcal B^q_{i,\link,\sigma}$\textup{)}. For two such collections indexed by the same set
$\mathcal I_i$, define
\[
d_{\mathrm{coll}}\big(\mathcal B_i^q(f),\mathcal B_i^q(g)\big)\ :=\
\max_{\tau\in\mathcal I_i}\ d_b\big(\mathcal B^q_{i,\tau}(f),\mathcal B^q_{i,\tau}(g)\big).
\]
The maximum exists because $\mathcal I_i$ is finite; as $\mathcal I_i$ is fixed
\textup{(}it depends on $\Delta$ and $i$, not on $f$ or $g$\textup{)}, $d_{\mathrm{coll}}$ is a metric
on $\mathcal I_i$-indexed barcode collections.
\end{definition}

\begin{theorem}[Stability of per-vertex local-cohomology barcodes]\label{thm:stability}
Let $f,g:\Delta\to\RR$ be monotonic with induced filtrations $(\Delta^t_f)$, $(\Delta^t_g)$. For
every vertex $i$, degree $q$, and multidegree $\alpha\in\Zn$,
\[
d_b\!\big(\mathcal B^q_{i,\alpha}(f),\,\mathcal B^q_{i,\alpha}(g)\big)\ \le\ \|f-g\|_\infty,
\]
and likewise for each of the face-indexed barcodes $\mathcal B^q_{i,\del,\sigma}$ and
$\mathcal B^q_{i,\link,\sigma}$ comprising $\mathcal B^q_i$ (Definition~\ref{def:barcode}). Here $d_b$
is the bottleneck distance.
\end{theorem}

\begin{proof}
Fix $i,q$ and a face $\sigma$; set $d=q-|\sigma|-1$ and $\delta=\|f-g\|_\infty$. For a complex
$\Delta'$ on $[n]$ write $K_{\del}(\Delta')=\link_{\Delta'}(\sigma)\cap\del_{\Delta'}(i)$ and
$K_{\link}(\Delta')=\link_{\Delta'}(\sigma\cup\{i\})$, the two subcomplexes of
Theorem~\ref{thm:persistent-hochster} (Corollary~\ref{cor:static-comb}). We treat the deletion barcode;
the link barcode is identical with $K_{\del}$ replaced by $K_{\link}$.

By the cochain-level naturality of Theorem~\ref{thm:persistent-hochster} (Step~1 of its proof),
$\mathbb V^q_{i,(0;-\mathbf 1_\sigma)}(f)$ is isomorphic, as an $\RR^{\mathrm{op}}$-indexed persistence
module, to $\big(\rH^{d}(K_{\del}(\Delta^t_f);\kk)\big)_t$ with the restriction maps of the inclusions
$K_{\del}(\Delta^s_f)\subseteq K_{\del}(\Delta^t_f)$; isomorphic persistence modules have equal
barcodes. Over the field $\kk$ this reduced-cohomology module is the $\kk$-linear dual of the
$\RR$-indexed reduced-homology module $\mathbf H_f:=\big(\rH_{d}(K_{\del}(\Delta^t_f);\kk)\big)_t$
(universal coefficients, naturally in the inclusion, as in the proof of
Theorem~\ref{thm:persistent-hochster}); dualization over $\kk$ reverses arrows and preserves interval supports, and hence barcodes, so it suffices to bound
$d_b\bigl(\mathrm{dgm}(\mathbf H_f),\mathrm{dgm}(\mathbf H_g)\bigr)$.

The operation $\Delta'\mapsto K_{\del}(\Delta')$ preserves inclusions of complexes on $[n]$ (link of the
fixed face $\sigma$, deletion of the fixed vertex $i$, and intersection are each monotone). Applying it
to the chain $\Delta^t_f\subseteq\Delta^{t+\delta}_g\subseteq\Delta^{t+2\delta}_f$ of
Lemma~\ref{lem:interleave-filt},
\[
K_{\del}(\Delta^t_f)\subseteq K_{\del}(\Delta^{t+\delta}_g)\subseteq K_{\del}(\Delta^{t+2\delta}_f)
\qquad(t\in\RR),
\]
and symmetrically with $f,g$ swapped, so the filtrations $(K_{\del}(\Delta^t_f))_t$ and
$(K_{\del}(\Delta^t_g))_t$ are $\delta$-interleaved. Since $\rH_d(-;\kk)$ is a functor and interleavings
are preserved by functors \cite{BubenikScott}, $\mathbf H_f$ and $\mathbf H_g$ are $\delta$-interleaved.
Both are of finite type, hence $q$-tame: each $K_{\del}(\Delta^t)$ is a subcomplex of the finite complex
$\Delta$, and the filtration changes at finitely many values of $t$. By the isometry theorem
(Theorem~\ref{thm:cdsgo}),
\[
d_b\big(\mathcal B^q_{i,\del,\sigma}(f),\mathcal B^q_{i,\del,\sigma}(g)\big)
=
d_b\bigl(\mathrm{dgm}(\mathbf H_f),\mathrm{dgm}(\mathbf H_g)\bigr)
\le\delta.
\]
The same argument with $K_{\link}$ gives
$d_b(\mathcal B^q_{i,\link,\sigma}(f),\mathcal B^q_{i,\link,\sigma}(g))\le\delta$. Taking the maximum
over the finitely many faces $\sigma\in\del_\Delta(i)$, resp.\ $\link_\Delta(i)$
(Proposition~\ref{prop:reduction}), gives the bound for every member of $\mathcal B^q_i$ and, in
particular, $d_{\mathrm{coll}}(\mathcal B^q_i(f),\mathcal B^q_i(g))\le\|f-g\|_\infty$.
\end{proof}

\begin{remark}[Stability of the face-indexed barcode collection]\label{rem:stability-meaning}
Theorem~\ref{thm:stability} gives, for every $\tau\in\mathcal I_i$,
\[
d_b\big(\mathcal B^q_{i,\tau}(f),\mathcal B^q_{i,\tau}(g)\big)\ \le\ \|f-g\|_\infty .
\]
Taking the maximum over the finite common index set $\mathcal I_i$ yields
\[
d_{\mathrm{coll}}\big(\mathcal B_i^q(f),\mathcal B_i^q(g)\big)\ \le\ \|f-g\|_\infty .
\]
Hence, for fixed $\Delta$, $i$, and $q$, the assignment $f\mapsto\mathcal B_i^q(f)$ is $1$-Lipschitz
from the monotonic functions $\Delta\to\RR$, with the sup-norm, to $\mathcal I_i$-indexed barcode
collections with $d_{\mathrm{coll}}$. In particular the per-vertex barcode is a robust descriptor of
$(\Delta,f)$: if $f$ is known only up to sup-norm error $\varepsilon$, any admissible $g$ with
$\|f-g\|_\infty\le\varepsilon$ yields a collection within $d_{\mathrm{coll}}$-distance $\varepsilon$.
This is the same $1$-Lipschitz behavior as for the Tor-side invariants of \cite{SuwayyidWeiGraphs}.

\end{remark}

Alternatively, Theorem~\ref{thm:summand-iso} transports the same stability bound from the interleaved deletion and link filtrations, with the same constant $\delta$.

\section{Attachment persistence and the \texorpdfstring{$x_i$}{xi}-action}
\label{sec:attachment}

The preceding results describe the deletion and link summands of
Theorem~\ref{thm:structure} separately, identify their persistent ranks through
Theorem~\ref{thm:persistent-hochster}, and associate to them the face-indexed
barcodes of Definition~\ref{def:barcode}.  Proposition~\ref{prop:xi-action}
shows, however, that these two summands are not independent as parts of the
original $S$-module: multiplication by $x_i$ maps the $x_i$-degree-$0$
deletion summand to the $x_i$-degree-$1$ link summand.

Geometrically, this coupling reflects the standard decomposition
\[
\Delta^t
=
\del_{\Delta^t}(i)
\cup
\st_{\Delta^t}(i),
\qquad
\del_{\Delta^t}(i)
\cap
\st_{\Delta^t}(i)
=
\link_{\Delta^t}(i).
\]

Motivated by this description, we refer to the persistence of the map coupling
the deletion and link summands as \emph{attachment persistence}.  The
attachment map does not replace the two separate persistence modules; rather,
it records the additional information that describes how they are coupled within
the $S$-module structure.

\subsection{The attachment morphism}
\label{ssec:attachment-map}

Fix a vertex $i$, a cohomological degree $q$, and a face
$\sigma\in\link_\Delta(i)$, and put
\(
d=q-|\sigma|-1.
\)
For every $t$, set
\begin{equation}
\label{eq:attachment-complexes}
A_{i,\sigma}^t
:=
\link_{\Delta^t}(\sigma)\cap\del_{\Delta^t}(i),
\qquad
B_{i,\sigma}^t
:=
\link_{\Delta^t}(\sigma\cup\{i\}).
\end{equation}
Equivalently, $A_{i,\sigma}^t=\link_{\del_{\Delta^t}(i)}(\sigma)$ and
$B_{i,\sigma}^t=\link_{\link_{\Delta^t}(i)}(\sigma)$, with the same
void-complex convention used in Theorem~\ref{thm:persistent-hochster}.
For every $t$, one has $B_{i,\sigma}^t\subseteq A_{i,\sigma}^t$.

By Theorem~\ref{thm:persistent-hochster}, the values of the corresponding
deletion and link persistence modules may be identified as
\begin{equation}
\label{eq:attachment-DL}
D_{t,\sigma}^q
:=
\bigl(\mathbb D^q_{i,\sigma}\bigr)_t
\cong
\rH^d(A_{i,\sigma}^t;\kk),
\qquad
L_{t,\sigma}^q
:=
\bigl(\mathbb L^q_{i,\sigma}\bigr)_t
\cong
\rH^d(B_{i,\sigma}^t;\kk).
\end{equation}
The inclusion \(B_{i,\sigma}^t\subseteq A_{i,\sigma}^t\) therefore induces a restriction map
\begin{equation}
\label{eq:attachment-map}
\rho_{t,\sigma}^q
\colon
D_{t,\sigma}^q
\longrightarrow
L_{t,\sigma}^q.
\end{equation}

Under the identifications \eqref{eq:attachment-DL}, this restriction map is
precisely the combinatorial realization of the $x_i$-action described in
Proposition~\ref{prop:xi-action}. Indeed, the multiplication of $x_i$ from
multidegree $(0;-\mathbf 1_\sigma)$ to $(1;-\mathbf 1_\sigma)$ is induced by
the quotient
\[
\kk[\del_{\Delta^t}(i)]
\twoheadrightarrow
\kk[\link_{\Delta^t}(i)].
\]
Restricting to multidegree $-\mathbf 1_\sigma$ and applying the natural
Hochster identification used in Theorem~\ref{thm:persistent-hochster}, this
quotient corresponds to the cohomology restriction induced by
\(
\link_{\link_{\Delta^t}(i)}(\sigma)
\subseteq
\link_{\del_{\Delta^t}(i)}(\sigma),
\)
that is,
\(
B_{i,\sigma}^t\subseteq A_{i,\sigma}^t.
\)
Thus multiplication by $x_i$ on this squarefree multidegree is exactly
\[
\rho_{t,\sigma}^q
\colon
\rH^d(A_{i,\sigma}^t;\kk)
\longrightarrow
\rH^d(B_{i,\sigma}^t;\kk).
\]

The naturality established in Proposition~\ref{prop:xi-action} gives, for
every $s\leq t$,
\begin{equation}
\label{eq:rho-naturality}
\rho_{s,\sigma}^q\circ\eta^{t\to s}
=
\zeta^{t\to s}\circ\rho_{t,\sigma}^q.
\end{equation}
Consequently, the maps $\rho_{t,\sigma}^q$ assemble into a natural
transformation of $\RR^{\mathrm{op}}$-indexed persistence modules
\begin{equation}
\label{eq:rho-natural-transformation}
\boldsymbol\rho^q_{i,\sigma}
\colon
\mathbb D^q_{i,\sigma}
\longrightarrow
\mathbb L^q_{i,\sigma}.
\end{equation}

The natural transformation $\boldsymbol\rho^q_{i,\sigma}$ retains the
deletion and link persistence modules together with the additional
$x_i$-coupling between them.  Forgetting the arrow recovers
$\mathbb D^q_{i,\sigma}$ and $\mathbb L^q_{i,\sigma}$, whereas the two
modules, or their barcodes considered separately, need not determine
$\boldsymbol\rho^q_{i,\sigma}$.

\subsection{Relation with relative cohomology}
\label{ssec:attachment-relative}

The morphism
\[
\boldsymbol\rho^q_{i,\sigma}\colon
\mathbb D^q_{i,\sigma}\longrightarrow
\mathbb L^q_{i,\sigma}
\]
is the primary attachment datum.  Since
$B_{i,\sigma}^t\subseteq A_{i,\sigma}^t$ at every filtration level, it is
also the middle morphism in the usual reduced-cohomology long exact sequence
of the pair.  Writing
\[
\mathbb R^q_{i,\sigma}
:=
\Bigl(
\widetilde H^{\,q-|\sigma|-1}
(A_{i,\sigma}^t,B_{i,\sigma}^t;\kk)
\Bigr)_{t\in\RR},
\]
naturality of relative cohomology gives a long exact sequence of
$\RR^{\mathrm{op}}$-indexed persistence modules
\[
\cdots\longrightarrow
\mathbb R^q_{i,\sigma}
\longrightarrow
\mathbb D^q_{i,\sigma}
\xrightarrow{\boldsymbol\rho^q_{i,\sigma}}
\mathbb L^q_{i,\sigma}
\longrightarrow
\mathbb R^{q+1}_{i,\sigma}
\longrightarrow\cdots .
\]

Thus the kernel and cokernel of $\boldsymbol\rho^q_{i,\sigma}$ are the
corresponding images in this standard long exact sequence.  We will use only
the image persistence module
\[
\mathbb I^q_{i,\sigma}
:=
\operatorname{im}\boldsymbol\rho^q_{i,\sigma},
\]
whose value at $t$ is
\[
\bigl(\mathbb I^q_{i,\sigma}\bigr)_t
=
\rho_{t,\sigma}^q(D_{t,\sigma}^q)
\subseteq L_{t,\sigma}^q.
\]
It records the part of the deletion persistence carried by multiplication
with $x_i$ into the link persistence.

\subsection{The coupled attachment rank}
\label{ssec:attachment-rank}

The persistent numbers introduced earlier measure the deletion and link
summands separately across two filtration levels.  The attachment map gives a
corresponding two-time quantity that couples the two summands.

\begin{definition}
\label{def:attachment-rank}
For $s\leq t$, define
\begin{equation}
\label{eq:attachment-rank-definition}
\chi^{q,s\to t}_{i,\sigma}
:=
\rank\Bigl(
D_{t,\sigma}^q
\xrightarrow{\rho_{t,\sigma}^q}
L_{t,\sigma}^q
\xrightarrow{\zeta^{t\to s}}
L_{s,\sigma}^q
\Bigr).
\end{equation}
By \eqref{eq:rho-naturality}, equivalently,
\[
\chi^{q,s\to t}_{i,\sigma}
=
\rank\Bigl(
D_{t,\sigma}^q
\xrightarrow{\eta^{t\to s}}
D_{s,\sigma}^q
\xrightarrow{\rho_{s,\sigma}^q}
L_{s,\sigma}^q
\Bigr).
\]
\end{definition}

\begin{corollary}
\label{thm:attachment-rank}
For $s\leq t$ and $d=q-|\sigma|-1$,
\begin{equation}
\label{eq:attachment-rank-cohomology}
\chi^{q,s\to t}_{i,\sigma}
=
\rank\Bigl(
\widetilde H^d(A_{i,\sigma}^t;\kk)
\longrightarrow
\widetilde H^d(B_{i,\sigma}^s;\kk)
\Bigr),
\end{equation}
where the map is reduced-cohomology restriction induced by
$B_{i,\sigma}^s\subseteq A_{i,\sigma}^t$.
Moreover,
\[
\chi^{q,s\to t}_{i,\sigma}
=
\rank\Bigl(
(\mathbb I^q_{i,\sigma})_t
\longrightarrow
(\mathbb I^q_{i,\sigma})_s
\Bigr).
\]
Hence $\chi^{q,s\to t}_{i,\sigma}$ is the rank invariant of the persistent
image of multiplication by $x_i$.
\end{corollary}

\begin{proof}
Since
\[
B_{i,\sigma}^s
\subseteq
B_{i,\sigma}^t
\subseteq
A_{i,\sigma}^t,
\]
the composite defining $\chi^{q,s\to t}_{i,\sigma}$ is restriction along
$B_{i,\sigma}^s\subseteq A_{i,\sigma}^t$.  The second assertion follows from
$\mathbb I^q_{i,\sigma}=\operatorname{im}\boldsymbol\rho^q_{i,\sigma}$.
\end{proof}

The attachment morphism can distinguish two situations in which the
corresponding deletion and link persistence modules have identical barcodes.
Thus the arrow-valued invariant contains information that is absent from the
two barcodes considered separately.

\begin{example}
\label{ex:attachment-strict}
Let $A_1=B_1=\partial(abc)$, the boundary of the triangle $abc$, and let
$B_2=\partial(abc)$ and
$A_2=\langle abc\rangle\cup\partial(ade)$, where $\langle abc\rangle$
is the filled $2$-simplex on $\{a,b,c\}$ and $\partial(ade)$ is the
boundary of a second triangle attached to the first at $a$. Then
$A_1\simeq A_2\simeq S^1$ and $B_1\cong B_2\cong S^1$.

Give every face filtration value $0$. Add a new vertex $i$ and set
$\Delta_r:=A_r\cup(i*B_r)$ for $r=1,2$. Then
$\del_{\Delta_r}(i)=A_r$ and $\link_{\Delta_r}(i)=B_r$. For
$\sigma=\varnothing$ and $q=2$, so that $d=1$, the corresponding deletion
and link persistence modules therefore have the same barcodes for $r=1$ and
$r=2$.

The attachment maps are different. For $r=1$, the map
$H^1(A_1;\kk)\to H^1(B_1;\kk)$ is an isomorphism. For $r=2$, the unique
nonzero class in $H^1(A_2;\kk)$ is represented by the cycle $ade$, while
$abc$ bounds the filled triangle $\langle abc\rangle$. Hence
$H^1(A_2;\kk)\to H^1(B_2;\kk)$ is the zero map.  Let
$\rho^{(r),2}_{i,\varnothing}$ denote the attachment morphism associated
with $\Delta_r$, for $r=1,2$.  Then
\(
\rank\rho^{(1),2}_{i,\varnothing}=1,
\) and \(
\rank\rho^{(2),2}_{i,\varnothing}=0.
\)

Equivalently, multiplication by $x_i$ from the deletion slice
$(0;\mathbf 0)$ to the link slice $(1;\mathbf 0)$ has different ranks in the
two examples, although the corresponding deletion and link barcodes agree.
Thus the arrow
\[
\boldsymbol\rho^q_{i,\sigma}\colon
\mathbb D^q_{i,\sigma}\longrightarrow
\mathbb L^q_{i,\sigma}
\]
contains information that is not determined by its source and target
persistence modules, and hence not by their barcodes.
\end{example}

\subsection{Stability of the attachment morphism}
\label{ssec:attachment-stability}

The attachment morphism inherits the same stability constant as the deletion
and link persistence modules.

\begin{corollary}[Stability of the $x_i$-action]
\label{thm:attachment-stability}
Let $f,g\colon\Delta\to\RR$ be monotonic and put
$\delta=\|f-g\|_\infty$.  For every $i,q,\sigma$, the arrows
\[
\boldsymbol\rho^q_{i,\sigma}(f)\colon
\mathbb D^q_{i,\sigma}(f)\longrightarrow
\mathbb L^q_{i,\sigma}(f)
\]
and
\[
\boldsymbol\rho^q_{i,\sigma}(g)\colon
\mathbb D^q_{i,\sigma}(g)\longrightarrow
\mathbb L^q_{i,\sigma}(g)
\]
are $\delta$-interleaved in the arrow category of
$\RR^{\mathrm{op}}$-indexed persistence modules.  Consequently, their image
modules are $\delta$-interleaved, and the corresponding image barcodes satisfy
\[
d_b\Bigl(
\mathcal B(\mathbb I^q_{i,\sigma}(f)),
\mathcal B(\mathbb I^q_{i,\sigma}(g))
\Bigr)
\leq
\|f-g\|_\infty.
\]
\end{corollary}

\begin{proof}
The interleaving inclusions of Lemma~\ref{lem:interleave-filt} remain
inclusions after applying
\[
\Delta'\longmapsto
A_{i,\sigma}(\Delta')
\qquad\text{and}\qquad
\Delta'\longmapsto
B_{i,\sigma}(\Delta'),
\]
and commute with $B_{i,\sigma}(\Delta')\subseteq A_{i,\sigma}(\Delta')$.
Reduced cohomology therefore gives a $\delta$-interleaving of the two
attachment arrows.  Passing to their images preserves the interleaving, and
the bottleneck estimate follows from Theorem~\ref{thm:cdsgo}.
\end{proof}

\section{Example}\label{sec:example}

We compute the face-$\varnothing$ per-vertex barcodes of a filtered disk on four vertices whose
filtration function admits no nontrivial symmetry. The four vertices then receive pairwise distinct
barcodes, and the top of the cone is the only vertex carrying an essential bar in cohomological
degree $2$. Both features are invisible to the persistent homology of $\Delta^\bullet$ itself, which
is a single barcode attached to no vertex.

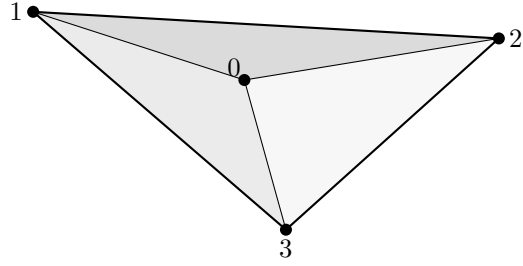
\begin{figure}[h]
\centering
\begin{tikzpicture}[scale=2.2,
    vtx/.style={circle,fill=black,inner sep=1.6pt},
    vlbl/.style={font=\small}]
\coordinate (v1) at (-1.40,0.36);
\coordinate (v2) at (1.40,0.20);
\coordinate (v3) at (0.12,-0.95);
\coordinate (v0) at (-0.13,-0.05);
\fill[black!14] (v0)--(v1)--(v2)--cycle;
\fill[black!8]  (v0)--(v1)--(v3)--cycle;
\fill[black!3]  (v0)--(v2)--(v3)--cycle;
\draw[thick] (v1)--(v2)--(v3)--(v1);
\draw (v0)--(v1); \draw (v0)--(v2); \draw (v0)--(v3);
\foreach \p in {v0,v1,v2,v3} \node[vtx] at (\p){};
\node[vlbl,above left=-2.5pt] at (v0) {$0$};
\node[vlbl,left] at (v1) {$1$};
\node[vlbl,right] at (v2) {$2$};
\node[vlbl,below] at (v3) {$3$};
\end{tikzpicture}
\caption{The complex $\Delta$ of Example~\ref{ex:cone}, drawn in perspective as the cone it is, seen
from above with the apex $0$ behind the rim cycle $1$--$2$--$3$ (heavy). The three shaded triangles
$012$, $023$, $013$ are faces; the rim triangle $\{1,2,3\}$ and the tetrahedron $\{0,1,2,3\}$ are
not, so $\Delta$ is a triangulated disk with boundary cycle $1$--$2$--$3$. The drawing records only
which subsets of $\{0,1,2,3\}$ are faces; no metric is implied, and it carries no filtration data.
The filtration function $f$ is given in the table below.}
\label{fig:cone}
\end{figure}

\begin{example}[A cone point, detected per vertex]\label{ex:cone}
We index the vertices by $\{0,1,2,3\}$, a relabeling of $[4]$. Let $\Delta$ be the cone with top $0$
over the boundary of the triangle on $\{1,2,3\}$: its faces are
the four vertices, the rim edges $12,23,31$, the top edges $01,02,03$, and the side triangles
$012,023,013$. The rim triangle $\{1,2,3\}$ and the tetrahedron are not faces, so $\Delta$ is a
triangulated disk with boundary cycle $1$--$2$--$3$. Filter it by the monotonic function $f$ with
$f(\{v\})=0$ for every vertex $v$ and
\[
\renewcommand{\arraystretch}{1.15}
\begin{array}{c|cccccc|ccc}
\tau & 23 & 01 & 02 & 13 & 12 & 03 & 012 & 023 & 013\\\hline
f(\tau) & 1 & 2 & 3 & 4 & 5 & 6 & 7 & 8 & 9
\end{array}
\]
(Figure~\ref{fig:cone}). Monotonicity holds because each side triangle is assigned a value larger
than those of its three edges. The filtration is not flag: at $t=5$ all three rim edges are present,
but $\{1,2,3\}$ is never filled, which is what keeps $\Delta$ a disk. Every automorphism of $\Delta$
fixes $0$, because $0$ is the only vertex whose link is a cycle, and permutes $\{1,2,3\}$; since $f$
takes distinct values on the three rim edges, the identity is the only automorphism of $\Delta$
preserving $f$. The asymmetry is therefore a property of the pair $(\Delta,f)$.

We record the face $\sigma=\varnothing$; the remaining faces are handled similarly. For
$\sigma=\varnothing$, Corollary~\ref{cor:static-comb} gives
$\link_{\link_{\Delta^t}(i)}(\varnothing)=\link_{\Delta^t}(i)$ and
$\link_{\del_{\Delta^t}(i)}(\varnothing)=\del_{\Delta^t}(i)$, so by the homological form of
Theorem~\ref{thm:persistent-hochster} the barcodes $\mathcal B^q_{i,\link,\varnothing}$ and
$\mathcal B^q_{i,\del,\varnothing}$ are those of ordinary persistent reduced homology of the
filtrations $(\link_{\Delta^t}(i))_t$, resp.\ $(\del_{\Delta^t}(i))_t$, in degree $q-1$.

Both filtrations are read off $f$ directly: a face $\tau$ with $i\notin\tau$ lies in
$\link_{\Delta^t}(i)$ exactly when $f(\tau\cup\{i\})\le t$, and in $\del_{\Delta^t}(i)$ exactly when
$\tau\in\Delta$ and $f(\tau)\le t$. In the table below, each entry $\tau{:}t$ names a face $\tau$,
set in bold, together with the value $t$ at which it enters; the vertices of the deletions, all of
which enter at $t=0$, are omitted.
\[
\renewcommand{\arraystretch}{1.3}
\begin{array}{c|ll|ll}
 & \multicolumn{2}{c|}{\link_{\Delta^t}(i)} & \multicolumn{2}{c}{\del_{\Delta^t}(i)}\\
i & \text{vertices} & \text{edges} & \text{edges} & \text{triangle}\\\hline
0 & \mathbf 1{:}2,\ \mathbf 2{:}3,\ \mathbf 3{:}6 & \mathbf{12}{:}7,\ \mathbf{23}{:}8,\ \mathbf{13}{:}9
  & \mathbf{23}{:}1,\ \mathbf{13}{:}4,\ \mathbf{12}{:}5 & \text{none}\\
1 & \mathbf 0{:}2,\ \mathbf 3{:}4,\ \mathbf 2{:}5 & \mathbf{02}{:}7,\ \mathbf{03}{:}9
  & \mathbf{23}{:}1,\ \mathbf{02}{:}3,\ \mathbf{03}{:}6 & \mathbf{023}{:}8\\
2 & \mathbf 3{:}1,\ \mathbf 0{:}3,\ \mathbf 1{:}5 & \mathbf{01}{:}7,\ \mathbf{03}{:}8
  & \mathbf{01}{:}2,\ \mathbf{13}{:}4,\ \mathbf{03}{:}6 & \mathbf{013}{:}9\\
3 & \mathbf 2{:}1,\ \mathbf 1{:}4,\ \mathbf 0{:}6 & \mathbf{02}{:}8,\ \mathbf{01}{:}9
  & \mathbf{01}{:}2,\ \mathbf{02}{:}3,\ \mathbf{12}{:}5 & \mathbf{012}{:}7
\end{array}
\]
Each link of a rim vertex is missing one edge: $\link_{\Delta^t}(1)$ never contains $23$, since
$\{1,2,3\}\notin\Delta$, and similarly for $2$ and $3$. Each link filtration is a filtration of a
complex on three vertices with at most two edges, and each deletion filtration is a filtration of a
graph on three vertices, possibly completed by one triangle.

The resulting barcodes are
\[
\renewcommand{\arraystretch}{1.25}
\begin{array}{c|ccc|ccc}
 & \multicolumn{3}{c|}{\text{link: }\mathcal B^q_{i,\link,\varnothing}}
 & \multicolumn{3}{c}{\text{deletion: }\mathcal B^q_{i,\del,\varnothing}}\\
i & q=0 & q=1 & q=2 & q=0 & q=1 & q=2\\\hline
0 & [0,2) & [3,7),\,[6,8) & [9,\infty) & \varnothing & [0,1),\,[0,4) & [5,\infty)\\
1 & [0,2) & [4,9),\,[5,7) & \varnothing & \varnothing & [0,1),\,[0,3) & [6,8)\\
2 & [0,1) & [3,8),\,[5,7) & \varnothing & \varnothing & [0,2),\,[0,4) & [6,9)\\
3 & [0,1) & [4,9),\,[6,8) & \varnothing & \varnothing & [0,2),\,[0,3) & [5,7)
\end{array}
\]
(Figure~\ref{fig:cone-barcode}); all degrees $q\ge3$ are empty as well, and the empty
deletion component in degree $q=0$ is the zero-length bar of Proposition~\ref{cor:minus-one-bars}
with $\sigma=\varnothing$, whose endpoints are both $0$. Three of the four rows of this table
are supplied by the general results of \S\ref{ssec:finite-type}. The degree-$0$ link bars are the case
$d=-1$ of Proposition~\ref{cor:minus-one-bars} with $\sigma=\varnothing$ and $\tau=\{i\}$: the bar is
\[
\Bigl[\,0,\ \min_{v\neq i,\ \{i,v\}\in\Delta}f(\{i,v\})\Bigr),
\]
so it dies when $i$ acquires its first edge, at $t=2$ for $i=0,1$ and at $t=1$ for $i=2,3$. The degree-$1$ deletion bars are the case $d=0$ of
Proposition~\ref{cor:msf-bars}, whose hypothesis holds because every vertex of a deletion enters at
$t=0$: for $\del_{\Delta^t}(0)$ the minimum spanning forest of the weighted graph
$23{:}1,\ 13{:}4,\ 12{:}5$ consists of $23$ and $13$, giving $[0,1)$ and $[0,4)$, and the graph is
connected, so there is no essential bar; the other three deletions are computed the same way. The
degree-$1$ link bars follow from the elder rule applied to the same tables, the vertices of a link now
entering at distinct times; Proposition~\ref{cor:msf-bars} does not apply to them for that reason. The
degree-$2$ bars record $\rH^1$ and are read off directly: $\rH^1(\del_{\Delta^t}(0);\kk)=\kk$ for
$t\ge5$, when the rim cycle closes and is never filled, and $\rH^1(\link_{\Delta^t}(0);\kk)=\kk$ for
$t\ge9$, when the link of the top becomes the rim cycle; for $i=1,2,3$ the cycle in
$\del_{\Delta^t}(i)$ is eventually filled by the triangle $\{0\}\cup(\{1,2,3\}\setminus\{i\})$ and every
link is a path, so no degree-$2$ bar is essential.

The four columns are pairwise distinct. The degree-$0$ link bar separates $\{0,1\}$ from
$\{2,3\}$; within $\{0,1\}$ the top is distinguished by its essential bars, and within $\{2,3\}$ the
degree-$2$ deletion bars $[6,9)$ and $[5,7)$ differ.

\emph{In the $\sigma=\varnothing$ component, the top is the unique vertex with an essential
degree-$2$ bar.} By Theorem~\ref{thm:structure} its essential bars record a nonzero module,
$\Hloc{2}{\pp_0}(\kk[\Delta^9])\neq0$: in the $b=0$ slice, the $x_0$-degree-$0$ piece is
$\rH^1(\del_{\Delta^9}(0);\kk)=\kk$ and each positive-$x_0$-degree piece is
$\rH^1(\link_{\Delta^9}(0);\kk)=\kk$. This is a statement about the $\sigma=\varnothing$ slice only.
In nonempty-face slices rim vertices also carry essential degree-$2$ bars: for $i=1$ and
$\sigma=\{0,2\}$, a facet of $\link_{\Delta^9}(1)$, Corollary~\ref{cor:static-comb} gives
$\Hloc{2}{\pp_1}(\kk[\Delta^9])_{(c;-\mathbf 1_{\{0,2\}})}\cong\rH^{-1}(\{\varnothing\};\kk)=\kk$ for
every $c\ge1$, and $\{0,2\}$ enters $\link_{\Delta^t}(1)$ at $t=f(\{0,1,2\})=7$, so
Proposition~\ref{cor:minus-one-bars} gives the essential bar $[7,\infty)$ in
$\mathcal B^2_{1,\link,\{0,2\}}$; in particular $\Hloc{2}{\pp_1}(\kk[\Delta])\neq0$. What singles out
the top is confined to the $\sigma=\varnothing$ slice, where its link is a $1$-cycle while every rim
link is a path.

Finally, the attachment morphism of \S\ref{ssec:attachment-map} is nontrivial here in the following
sense. Take $i=0$, $\sigma=\varnothing$ and $q=2$, so $d=1$, $A^t_{0,\varnothing}=\del_{\Delta^t}(0)$
and $B^t_{0,\varnothing}=\link_{\Delta^t}(0)$. For $5\le t\le8$ one has
$\rH^1(A^t_{0,\varnothing};\kk)=\kk$ and $\rH^1(B^t_{0,\varnothing};\kk)=0$, so
$\rho^2_{t,\varnothing}=0$, whereas at $t=9$ the two complexes coincide with the rim cycle and
$\rho^2_{9,\varnothing}$ is an isomorphism. Hence
$\chi^{2,s\to t}_{0,\varnothing}=1$ if $s=t=9$ and $0$ otherwise, by
Corollary~\ref{thm:attachment-rank}: multiplication by $x_0$ carries the deletion class into the link
only after the top edges and side triangles have all appeared.
\end{example}

\begin{figure}[h]
\centering
\begin{tikzpicture}[x=0.28cm,y=1cm,font=\scriptsize,
 fin/.style={line width=1.0pt,line cap=butt},
 ess/.style={line width=1.0pt,line cap=butt,-{Latex[length=1.5mm]}}]
\node[anchor=east] at (-6.25,2.79) {link};
\node[anchor=east] at (-6.25,1.24) {del};
\begin{scope}[xshift=0.0000cm]
\node[anchor=south] at (4.5,3.42) {$i=0$\ (top)};
\node[anchor=east,inner sep=1.5pt] at (0,3.30) {$0$};
\draw[fin] (0,3.30)--(2,3.30);
\node[anchor=east,inner sep=1.5pt] at (0,2.96) {$1$};
\draw[fin] (3,2.96)--(7,2.96);
\node[anchor=east,inner sep=1.5pt] at (0,2.62) {$1$};
\draw[fin] (6,2.62)--(8,2.62);
\node[anchor=east,inner sep=1.5pt] at (0,2.28) {$2$};
\draw[ess] (9,2.28)--(10.7,2.28);
\draw[gray!35] (0,1.88)--(9.4,1.88);
\node[anchor=east,inner sep=1.5pt] at (0,1.58) {$1$};
\draw[fin] (0,1.58)--(1,1.58);
\node[anchor=east,inner sep=1.5pt] at (0,1.24) {$1$};
\draw[fin] (0,1.24)--(4,1.24);
\node[anchor=east,inner sep=1.5pt] at (0,0.90) {$2$};
\draw[ess] (5,0.90)--(10.7,0.90);
\draw[gray!55] (0,0.54)--(9.4,0.54);
\foreach \t in {0,...,9}{\draw[gray!55](\t,0.49)--(\t,0.59);}
\foreach \t in {0,3,6,9}{\node[gray!55,below,inner sep=1.2pt] at (\t,0.50){$\t$};}
\node[gray!55,right,inner sep=1.5pt] at (9.5,0.54){$t$};
\end{scope}
\begin{scope}[xshift=3.3500cm]
\node[anchor=south] at (4.5,3.42) {$i=1$};
\node[anchor=east,inner sep=1.5pt] at (0,3.30) {$0$};
\draw[fin] (0,3.30)--(2,3.30);
\node[anchor=east,inner sep=1.5pt] at (0,2.96) {$1$};
\draw[fin] (4,2.96)--(9,2.96);
\node[anchor=east,inner sep=1.5pt] at (0,2.62) {$1$};
\draw[fin] (5,2.62)--(7,2.62);
\draw[gray!35] (0,1.88)--(9.4,1.88);
\node[anchor=east,inner sep=1.5pt] at (0,1.58) {$1$};
\draw[fin] (0,1.58)--(1,1.58);
\node[anchor=east,inner sep=1.5pt] at (0,1.24) {$1$};
\draw[fin] (0,1.24)--(3,1.24);
\node[anchor=east,inner sep=1.5pt] at (0,0.90) {$2$};
\draw[fin] (6,0.90)--(8,0.90);
\draw[gray!55] (0,0.54)--(9.4,0.54);
\foreach \t in {0,...,9}{\draw[gray!55](\t,0.49)--(\t,0.59);}
\foreach \t in {0,3,6,9}{\node[gray!55,below,inner sep=1.2pt] at (\t,0.50){$\t$};}
\node[gray!55,right,inner sep=1.5pt] at (9.5,0.54){$t$};
\end{scope}
\begin{scope}[xshift=6.7000cm]
\node[anchor=south] at (4.5,3.42) {$i=2$};
\node[anchor=east,inner sep=1.5pt] at (0,3.30) {$0$};
\draw[fin] (0,3.30)--(1,3.30);
\node[anchor=east,inner sep=1.5pt] at (0,2.96) {$1$};
\draw[fin] (3,2.96)--(8,2.96);
\node[anchor=east,inner sep=1.5pt] at (0,2.62) {$1$};
\draw[fin] (5,2.62)--(7,2.62);
\draw[gray!35] (0,1.88)--(9.4,1.88);
\node[anchor=east,inner sep=1.5pt] at (0,1.58) {$1$};
\draw[fin] (0,1.58)--(2,1.58);
\node[anchor=east,inner sep=1.5pt] at (0,1.24) {$1$};
\draw[fin] (0,1.24)--(4,1.24);
\node[anchor=east,inner sep=1.5pt] at (0,0.90) {$2$};
\draw[fin] (6,0.90)--(9,0.90);
\draw[gray!55] (0,0.54)--(9.4,0.54);
\foreach \t in {0,...,9}{\draw[gray!55](\t,0.49)--(\t,0.59);}
\foreach \t in {0,3,6,9}{\node[gray!55,below,inner sep=1.2pt] at (\t,0.50){$\t$};}
\node[gray!55,right,inner sep=1.5pt] at (9.5,0.54){$t$};
\end{scope}
\begin{scope}[xshift=10.0500cm]
\node[anchor=south] at (4.5,3.42) {$i=3$};
\node[anchor=east,inner sep=1.5pt] at (0,3.30) {$0$};
\draw[fin] (0,3.30)--(1,3.30);
\node[anchor=east,inner sep=1.5pt] at (0,2.96) {$1$};
\draw[fin] (4,2.96)--(9,2.96);
\node[anchor=east,inner sep=1.5pt] at (0,2.62) {$1$};
\draw[fin] (6,2.62)--(8,2.62);
\draw[gray!35] (0,1.88)--(9.4,1.88);
\node[anchor=east,inner sep=1.5pt] at (0,1.58) {$1$};
\draw[fin] (0,1.58)--(2,1.58);
\node[anchor=east,inner sep=1.5pt] at (0,1.24) {$1$};
\draw[fin] (0,1.24)--(3,1.24);
\node[anchor=east,inner sep=1.5pt] at (0,0.90) {$2$};
\draw[fin] (5,0.90)--(7,0.90);
\draw[gray!55] (0,0.54)--(9.4,0.54);
\foreach \t in {0,...,9}{\draw[gray!55](\t,0.49)--(\t,0.59);}
\foreach \t in {0,3,6,9}{\node[gray!55,below,inner sep=1.2pt] at (\t,0.50){$\t$};}
\node[gray!55,right,inner sep=1.5pt] at (9.5,0.54){$t$};
\end{scope}
\end{tikzpicture}
\caption{The face-$\varnothing$ per-vertex barcodes of Example~\ref{ex:cone}: link summand (upper
block, labelled \emph{link}) and deletion summand (lower block, labelled \emph{del}) for each of the four vertices. Left labels in each block are
the cohomological degree $q$; arrows mark essential bars. The four columns are pairwise distinct, and
only the top $0$ has essential bars.}
\label{fig:cone-barcode}
\end{figure}
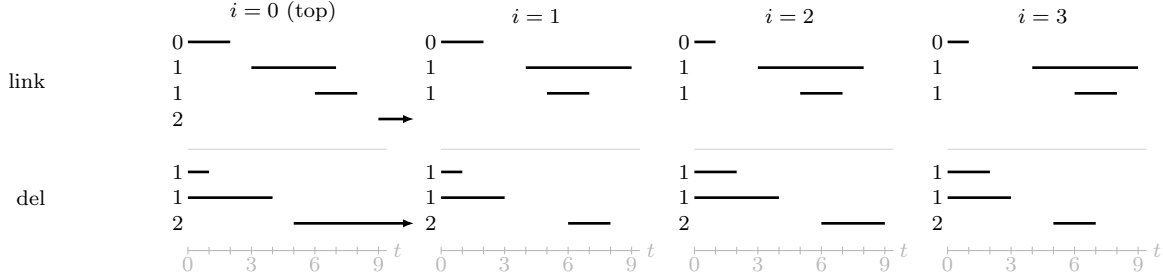

\section{Coordinate primes and Boolean multiplication diagrams}
\label{sec:coordinate-primes}

The preceding sections focused on a vertex prime
$\pp_i=(x_j:j\neq i)$, for which there is a single uninverted variable
$x_i$.  The resulting local-cohomology decomposition has two pieces,
corresponding to $x_i$-degree $0$ and positive $x_i$-degree, and the
attachment morphism records multiplication from the former to the latter.
The same mechanism extends naturally to an arbitrary coordinate prime.  With
several uninverted variables, their positive support may be any subset of a
coordinate set $U$, so the two-piece decomposition is replaced by a Boolean
family of summands and attachment morphisms.

Fix $U\subseteq[n]$, set $T=[n]\setminus U$, and write
\(
\pp_U:=(x_j:j\in T)\subseteq S\),
\(S_T:=\kk[x_j:j\in T]\), and
\(
\mm_T:=(x_j:j\in T).
\)
Then $\pp_U$ is a coordinate prime of the ambient polynomial ring $S$, and
\[
S/\pp_U\cong\kk[x_u:u\in U].
\]
Its image in the Stanley--Reisner ring $\kk[\Delta]=S/I_\Delta$ need not
be a prime ideal.  Under our standing convention that every singleton
$\{i\}$ is a face of $\Delta$, this image is prime if and only if
$U\in\Delta$.  Indeed,
\[
\kk[\Delta]/\pp_U\kk[\Delta]
\cong
\kk[\Delta_U],
\]
and $\Delta_U$ is the full simplex on $U$ precisely when $U\in\Delta$.

We nevertheless allow arbitrary $U\subseteq[n]$: throughout this section,
$\Hloc{q}{\pp_U}(\kk[\Delta])$ means local cohomology of the
$S$-module $\kk[\Delta]$ with support in the coordinate prime
$\pp_U\subseteq S$.
For $A\subseteq U$, define the simplicial complex on the ambient set $T$ by
\begin{equation}
\label{eq:gamma-UA}
\Gamma_A(\Delta)
:=
\{\tau\subseteq T:A\cup\tau\in\Delta\}
=
\bigl(\link_\Delta(A)\bigr)_T.
\end{equation}
When $A\notin\Delta$, this is the void complex and its Stanley--Reisner ring
is zero.  Moreover, $A\subseteq A'\subseteq U$ implies
$\Gamma_{A'}(\Delta)\subseteq\Gamma_A(\Delta)$.

The subsets $A\subseteq U$ form the Boolean lattice $2^U$.  Equivalently,
their indicator vectors form the vertices of the cube
$\{0,1\}^{|U|}$, and an inclusion $A\subset A\cup\{u\}$ is one of its
edges.  This is the sense in which the coordinate-prime structure below is
Boolean: multiplication by an uninverted variable $x_u$ moves between
summands along these edges.

For $A\subseteq U$, let
$M_A:=\bigotimes_{u\in A}x_u\kk[x_u]$, viewed as the $\kk$-span of
monomials in the $U$-variables having positive support exactly $A$;
coordinates in $U\setminus A$ have exponent zero.  Set
$M_\varnothing=\kk$.

\subsection{The coordinate-prime structure theorem}
\label{ssec:coordinate-structure}

\begin{theorem}
\label{thm:coordinate-chain-split}
Let $c\in\ZZ^U$, and let
$\check C_T^\bullet(\kk[\Delta])^{(c)}$ denote the part of the \v{C}ech
complex with fixed $U$-multidegree $c$.  If some coordinate of $c$ is
negative, then this complex is zero.  If $c\in\NN^U$ and
$A=\supp_+(c)$, there is a natural isomorphism of $\ZZ^T$-graded complexes
of $S_T$-modules
\begin{equation}
\label{eq:coordinate-chain-isomorphism}
\check C_T^\bullet(\kk[\Delta])^{(c)}
\cong
x^c\check C_{\mm_T}^\bullet(\kk[\Gamma_A(\Delta)]).
\end{equation}
Consequently,
\begin{equation}
\label{eq:coordinate-chain-direct-sum}
\check C_T^\bullet(\kk[\Delta])
\cong
\bigoplus_{A\subseteq U}
\check C_{\mm_T}^\bullet(\kk[\Gamma_A(\Delta)])
\otimes_\kk M_A
\end{equation}
as $\ZZ^n$-graded complexes of $S_T$-modules.
\end{theorem}

\begin{proof}
The \v{C}ech complex for $\pp_U$ inverts only the variables $x_j$ with
$j\in T$.  Hence a negative exponent in a $U$-coordinate cannot occur in
any localization.  Now let $c\in\NN^U$, put $A=\supp_+(c)$, and consider
a \v{C}ech summand indexed by $F\subseteq T$.  For $b\in\ZZ^T$,
Lemma~\ref{lem:localization} says that $x^cx^b$ is nonzero in
$\kk[\Delta]_{x_F}$ exactly when $b_j\geq0$ for $j\in T\setminus F$ and
$A\cup\supp_+(b)\cup F\in\Delta$.  By \eqref{eq:gamma-UA}, this is exactly
the condition that $x^b$ be nonzero in
$\kk[\Gamma_A(\Delta)]_{x_F}$.  The bijection
$x^cx^b\mapsto x^c\otimes x^b$ preserves the localization maps and
\v{C}ech signs, giving \eqref{eq:coordinate-chain-isomorphism}.  Summing
over $c\in\NN^U$ and grouping the exponent vectors by their positive
supports gives \eqref{eq:coordinate-chain-direct-sum}.
\end{proof}

\begin{theorem}
\label{thm:coordinate-structure}
For every $q\geq0$, there is an isomorphism of $\ZZ^n$-graded
$S_T$-modules
\begin{equation}
\label{eq:coordinate-structure}
\Hloc{q}{\pp_U}(\kk[\Delta])
\cong
\bigoplus_{A\subseteq U}
\Hloc{q}{\mm_T}(\kk[\Gamma_A(\Delta)])\otimes_\kk M_A.
\end{equation}
More explicitly, for $c\in\ZZ^U$ and $b\in\ZZ^T$,
\begin{equation}
\label{eq:coordinate-multidegree}
\Hloc{q}{\pp_U}(\kk[\Delta])_{(c;b)}
\cong
\begin{cases}
\Hloc{q}{\mm_T}(\kk[\Gamma_A(\Delta)])_b,
& c\in\NN^U,\ A=\supp_+(c),\\
0,& c\notin\NN^U.
\end{cases}
\end{equation}
For $u\in U$, multiplication by $x_u$ acts as follows:
\begin{enumerate}[label=\textup{(\roman*)},leftmargin=2.2em]
\item if $u\in A$, it is the ordinary exponent shift within the
$A$-summand;
\item if $u\notin A$, it is the map from the $A$-summand to the
$A\cup\{u\}$-summand induced by
$\kk[\Gamma_A(\Delta)]\twoheadrightarrow
\kk[\Gamma_{A\cup\{u\}}(\Delta)]$.
\end{enumerate}
These maps make \eqref{eq:coordinate-structure} an isomorphism of
$S=S_T[x_u:u\in U]$-modules.  For distinct
$u,v\in U\setminus A$, the square obtained by adjoining $u$ and $v$ in
either order commutes.
\end{theorem}

\begin{proof}
Taking cohomology in Theorem~\ref{thm:coordinate-chain-split} gives
\eqref{eq:coordinate-structure} and \eqref{eq:coordinate-multidegree}.  If
$u\in A$, multiplication by $x_u$ does not change the positive
$U$-support, so it shifts the exponent in $M_A$.  If $u\notin A$, the
positive support changes from $A$ to $A\cup\{u\}$.  At chain level,
multiplication retains a face $\tau\subseteq T$ precisely when
$A\cup\{u\}\cup\tau\in\Delta$; otherwise it sends the corresponding
monomial to zero.  This is exactly the \v{C}ech map induced by the quotient
in \textup{(ii)}.  Passing to cohomology proves the description of the
$x_u$-action.  Adjoining $u$ and $v$ in either order gives the same quotient
onto $\kk[\Gamma_{A\cup\{u,v\}}(\Delta)]$, so every Boolean square
commutes.
\end{proof}

For $U=\{i\}$, the two subsets are $\varnothing$ and $\{i\}$, with
$\Gamma_\varnothing(\Delta)=\del_\Delta(i)$ and
$\Gamma_{\{i\}}(\Delta)=\link_\Delta(i)$.  Thus
Theorem~\ref{thm:coordinate-structure} specializes to
Theorem~\ref{thm:structure}, and its unique Boolean edge is the
$x_i$-action of Proposition~\ref{prop:xi-action}.  In particular, the
attachment morphism of Section~\ref{sec:attachment} is the one-dimensional
Boolean case of the structure above.

At the level of fixed multigraded vector-space dimensions,
\eqref{eq:coordinate-multidegree} is a specialization of Rahimi's
formula for local cohomology with support in a monomial prime
\cite{Rahimi}.  The additional structure retained here is the natural
chain-level decomposition, the action of all uninverted variables,
their organization into the commuting Boolean multiplication diagram,
and the functoriality of this structure under simplicial filtrations.

\subsection{The coordinate-prime persistent Hochster formula}
\label{ssec:coordinate-persistent}

For a filtration $\Delta^\bullet$, write
$\Gamma_A^t:=\Gamma_A(\Delta^t)$.  For $s\leq t$, the quotient
$\kk[\Delta^t]\twoheadrightarrow\kk[\Delta^s]$ preserves
$U$-multidegrees, so Theorem~\ref{thm:coordinate-structure} is natural in
$t$.

\begin{theorem}
\label{thm:coordinate-persistent-hochster}
Fix $q$, $s\leq t$, $c\in\ZZ^U$, and $b\in\ZZ^T$.  Set
$\sigma=\supp_-(b)$ and $d=q-|\sigma|-1$.  The multigraded persistence
module vanishes if $c\notin\NN^U$ or
$\supp_+(b)\neq\varnothing$.  Otherwise, set $A=\supp_+(c)$.  There is a
natural isomorphism of $\RR^{\mathrm{op}}$-indexed persistence modules
\begin{equation}
\label{eq:coordinate-persistence-isomorphism}
\Bigl(\Hloc{q}{\pp_U}(\kk[\Delta^t])_{(c;b)}\Bigr)_t
\cong
\Bigl(
\rH^d\bigl((\link_{\Delta^t}(A\cup\sigma))_T;\kk\bigr)
\Bigr)_t.
\end{equation}
Consequently, its persistent rank from $t$ to $s$ is
\begin{equation}
\label{eq:coordinate-persistent-rank}
\rank\Bigl(
\rH^d\bigl((\link_{\Delta^t}(A\cup\sigma))_T;\kk\bigr)
\longrightarrow
\rH^d\bigl((\link_{\Delta^s}(A\cup\sigma))_T;\kk\bigr)
\Bigr).
\end{equation}
The map is the restriction in reduced cohomology induced by inclusion, and a
void link contributes zero.  Up to canonical isomorphism, only the
squarefree representatives $c=\mathbf 1_A$ and
$b=-\mathbf 1_\sigma$ are needed.
\end{theorem}

\begin{proof}
By Theorem~\ref{thm:coordinate-structure}, naturally in $t$, the left-hand
side reduces to
$\Hloc{q}{\mm_T}(\kk[\Gamma_A^t])_b$.  Hochster's formula identifies this
with $\rH^d(\link_{\Gamma_A^t}(\sigma);\kk)$ and gives the stated
vanishing.  From \eqref{eq:gamma-UA},
\[
\link_{\Gamma_A^t}(\sigma)
=
\{\tau\subseteq T\setminus\sigma:
A\cup\sigma\cup\tau\in\Delta^t\}
=
(\link_{\Delta^t}(A\cup\sigma))_T.
\]
The cochain-level Hochster isomorphism is natural for inclusions, as used in
Theorem~\ref{thm:persistent-hochster}; hence these levelwise
identifications form \eqref{eq:coordinate-persistence-isomorphism}.  The
rank formula is the structure-map rank of the resulting reversed
cohomology persistence module.  Neither side depends on the positive
magnitudes of $c$ or the negative magnitudes of $b$, only on $A$ and
$\sigma$.
\end{proof}

The explicit degree-$(-1)$ formula from
Proposition~\ref{cor:minus-one-bars} also extends directly to every
coordinate prime.

\begin{corollary}
\label{cor:coordinate-minus-one}
In Theorem~\ref{thm:coordinate-persistent-hochster}, suppose $d=-1$ and
$A\cup\sigma\in\Delta$.  The barcode consists of the single possible
interval
\begin{equation}
\label{eq:coordinate-minus-one}
\left[
f(A\cup\sigma),
\min_{\substack{v\in T\setminus\sigma\\
A\cup\sigma\cup\{v\}\in\Delta}}
f(A\cup\sigma\cup\{v\})
\right),
\end{equation}
with the same convention that the right endpoint is $\infty$ if the
displayed set of vertices is empty.
\end{corollary}

\begin{proof}
The complex $(\link_{\Delta^t}(A\cup\sigma))_T$ changes from void to
$\{\varnothing\}$ when $A\cup\sigma$ appears and acquires its first vertex
exactly at the right endpoint of \eqref{eq:coordinate-minus-one}.  The claim
therefore follows from the characterization of reduced
$(-1)$-cohomology used in Proposition~\ref{cor:minus-one-bars}.
\end{proof}

\subsection{The persistent Boolean multiplication diagram}
\label{ssec:boolean-attachment}

The coordinate-prime structure retains more than the persistence of the
individual summands.  For each $A\subseteq U$, each
$u\in U\setminus A$, and each relevant face $\sigma$, put
\[
K_{A,\sigma}^t
:=
(\link_{\Delta^t}(A\cup\sigma))_T.
\]
Then
\[
K_{A\cup\{u\},\sigma}^t
\subseteq
K_{A,\sigma}^t,
\]
and Theorems~\ref{thm:coordinate-structure} and
\ref{thm:coordinate-persistent-hochster} identify multiplication by $x_u$
with the natural transformation
\[
\boldsymbol\rho^q_{A,u,\sigma}\colon
\Bigl(\widetilde H^d(K_{A,\sigma}^t;\kk)\Bigr)_t
\longrightarrow
\Bigl(\widetilde H^d(K_{A\cup\{u\},\sigma}^t;\kk)\Bigr)_t,
\qquad
d=q-|\sigma|-1.
\]

As $A$ ranges over $2^U$, these morphisms form a commuting Boolean diagram of
$\RR^{\mathrm{op}}$-indexed persistence modules.  In particular, for distinct
$u,v\in U\setminus A$, the diagram
\[
\begin{tikzcd}
\mathbb P^q_{A,\sigma}
\arrow[r,"x_u"]
\arrow[d,"x_v"']
&
\mathbb P^q_{A\cup\{u\},\sigma}
\arrow[d,"x_v"]
\\
\mathbb P^q_{A\cup\{v\},\sigma}
\arrow[r,"x_u"']
&
\mathbb P^q_{A\cup\{u,v\},\sigma}
\end{tikzcd}
\]
commutes, where
\(
\mathbb P^q_{A,\sigma}
:=
\Bigl(\widetilde H^d(K_{A,\sigma}^t;\kk)\Bigr)_t.
\)

For $s\leq t$, the rank of the edge morphism from level $t$ to level $s$ is
\[
\chi^{q,s\to t}_{U,A,u,\sigma}
=
\rank\Bigl(
\widetilde H^d(K_{A,\sigma}^t;\kk)
\longrightarrow
\widetilde H^d(K_{A\cup\{u\},\sigma}^s;\kk)
\Bigr).
\]
Thus the coordinate-prime theory organizes the multiplication maps of all
uninverted variables simultaneously over the Boolean lattice, rather than
retaining the summand barcodes independently.

Each Boolean edge inherits the same stability bound as the vertex-prime
attachment morphism.  More precisely, if
$f,g\colon\Delta\to\RR$ are monotone filtration functions and
$\delta=\|f-g\|_\infty$, then, for every $A\subseteq U$,
$u\in U\setminus A$, $q$, and relevant face $\sigma$, the corresponding
edge morphisms
\(
\boldsymbol\rho^q_{A,u,\sigma}(f)
\) and \(
\boldsymbol\rho^q_{A,u,\sigma}(g)
\)
are $\delta$-interleaved in the arrow category of
$\RR^{\mathrm{op}}$-indexed persistence modules.  Indeed, the
interleaving inclusions for $\Delta_f^\bullet$ and $\Delta_g^\bullet$
remain inclusions after applying
\[
\Delta'\longmapsto
\bigl(\link_{\Delta'}(A\cup\sigma)\bigr)_T
\]
and are compatible with
\[
\bigl(\link_{\Delta'}(A\cup\{u\}\cup\sigma)\bigr)_T
\subseteq
\bigl(\link_{\Delta'}(A\cup\sigma)\bigr)_T.
\]
Thus the argument of Corollary~\ref{thm:attachment-stability}
applies verbatim to every Boolean edge.

For $U=\{i\}$ the Boolean lattice has two vertices and one edge, and the
diagram reduces exactly to
\(
\boldsymbol\rho^q_{i,\sigma}\colon
\mathbb D^q_{i,\sigma}\longrightarrow
\mathbb L^q_{i,\sigma}.
\)

\section{Relation to existing work and scope}\label{sec:scope}

We now situate the construction within the program of \cite{SuwayyidWeiPSRT,SuwayyidWeiGraphs}.

Theorem~\ref{thm:structure} reduces $\Hloc{\bullet}{\pp_i}(\kk[\Delta])$ to maximal-support local
cohomology of the link and deletion of $i$, and Corollary~\ref{cor:static-comb} expresses each graded
piece as reduced simplicial cohomology of an iterated link of $\Delta$ over the field $\kk$. This is
combinatorial data of the same kind---reduced homology of complexes derived from $\Delta$, with
$\kk$-coefficients---that Hochster's Tor-side formula attaches to the multigraded Betti numbers of
$\kk[\Delta]$ (there through induced subcomplexes rather than links). The per-vertex barcode is not a homological quantity completely distinct from the Tor-side persistent Betti numbers of \cite{SuwayyidWeiGraphs}: it persists the same kind of
reduced-homology data, reorganized on the injective side and localized at a single vertex. The new
content is therefore (i) the per-vertex localization $i\mapsto\pp_i$, which isolates the contribution
of one vertex; (ii) the explicit deletion/link decomposition of Theorem~\ref{thm:structure}; and (iii)
the resulting per-vertex barcodes and their bottleneck stability. The localization (i) is what lets the
barcodes attach a feature to a single vertex, as in Example~\ref{ex:cone}, where the top of a cone is
singled out by an essential degree-$2$ bar.

\paragraph{What this construction does not capture, and a sharper open problem.}

The face-indexed barcodes record the persistence of
$\dim_\kk\Hloc{q}{\mm'}(\kk[\del_{\Delta^t}(i)])_b$
and
$\dim_\kk\Hloc{q}{\mm'}(\kk[\link_{\Delta^t}(i)])_b$
separately---equivalently, of the pointwise vector-space persistence modules
$\mathbb V^q_{i,(c;b)}$ (Theorem~\ref{thm:summand-iso}(a)). They do not by
themselves determine the $\Zn$-graded $S$-module
$\Hloc{q}{\pp_i}(\kk[\Delta])$ itself. Two kinds of data are not contained in
the separate barcodes. First, the action of the remaining variable $x_i$: it
sends the $x_i$-degree-$0$ (deletion) part to the degree-$1$ (link) part
through the map on $\Hloc{q}{\mm'}(-)$ induced by the quotient
$\kk[\del_\Delta(i)]\twoheadrightarrow\kk[\link_\Delta(i)]$, whose rank is
not recovered from the two summands' barcodes. This first piece of missing
information is recovered in Section~\ref{sec:attachment}, where the
$x_i$-action is retained as the natural transformation
\[
\boldsymbol\rho^q_{i,\sigma}\colon
\mathbb D^q_{i,\sigma}\longrightarrow
\mathbb L^q_{i,\sigma}.
\]
For coordinate primes, Section~\ref{sec:coordinate-primes} organizes the
corresponding multiplication maps of the uninverted variables into a
commuting Boolean diagram. Second, the separate barcodes do not retain the
internal $S'$-multiplications between different multidegrees $b$, which
Gr\"abe realizes \cite{Grabe} as maps between reduced cohomologies of distinct
links, not determined by their Betti numbers alone. The barcodes therefore
classify the individual one-parameter modules $\mathbb V^q_{i,(c;b)}$, but
not by themselves the full multigraded module-valued persistence object.

Beyond this, the construction inherits exactly the characteristic dependence of ordinary Hochster
theory and no more; it does not see the finer characteristic-dependent injective invariants---Bass
numbers and Lyubeznik numbers---of $\kk[\Delta]$
\cite{AlvarezMontanerSohrabi,HelmMiller,AlvarezMontanerVahidi}, whose persistence lies genuinely
outside both the Tor-side framework and the present one. Persisting these across a filtration would
record information different from both the present barcodes and from the Tor-side persistent Betti numbers of
\cite{SuwayyidWeiGraphs}. We regard this as the most promising direction beyond the present paper.

The proposed localized persistent commutative algebra can be readily applied to a wide variety of problems in science and engineering. In particular, it can be used to characterize the global properties of a specific data point, such as an individual node within a network, as well as the local or internal properties of a specific data point, such as a drug molecule in the DrugBank database.



\section{Conclusion and outlook}\label{sec:conclusion}

We have localized the persistent Stanley--Reisner theory of
\cite{SuwayyidWeiPSRT,SuwayyidWeiGraphs} at a vertex prime where each cohomology module is splitting into a deletion and link summand. Then they are persisted across a filtration to provide the gives per-vertex barcodes
(Definition~\ref{def:barcode}), a persistent links--Hochster formula
(Theorem~\ref{thm:persistent-hochster}), interval decomposability
(Theorem~\ref{thm:finite-type}), bottleneck stability (Theorem~\ref{thm:stability})

Some possible future ideas including : 
\paragraph{Bass numbers.} The invariants persisted here are graded dimensions of
local cohomology, and Corollary~\ref{cor:static-comb} shows that they depend on the
field $\kk$ only through the reduced cohomology of links. The injective side carries
finer invariants that this dependence does not see. For a prime $\pp$ of a Noetherian
ring $R$ and an $R$-module $M$, the Bass numbers
\[
\mu^p(\pp,M)\;=\;\dim_{\kappa(\pp)}\Ext^p_{R_\pp}\bigl(\kappa(\pp),M_\pp\bigr),
\qquad \kappa(\pp)=R_\pp/\pp R_\pp,
\]
count the copies of the injective hull $E_R(R/\pp)$ in the minimal injective
resolution of $M$ \cite[\S3.2]{BrunsHerzog}; they are the injective-side analogue of
the graded Betti numbers persisted in \cite{SuwayyidWeiGraphs}. For local cohomology
of monomial ideals these have been computed combinatorially by Helm--Miller
\cite{HelmMiller} in the semigroup-graded setting and by \`Alvarez Montaner--Sohrabi
\cite{AlvarezMontanerSohrabi} for cover ideals of graphs. But a persistent theory of Bass
numbers does not follow from the present one; it requires a separate finiteness and
functoriality analysis.

\paragraph{Lyubeznik numbers.} Let $S=\kk[x_1,\dots,x_n]$, let $I\subseteq S$ be a
graded ideal, and set $R=S/I$ of dimension $d$. The Lyubeznik numbers
\[
\lambda_{p,q}(R)\;=\;\mu^p\bigl(\mm,\Hloc{n-q}{I}(S)\bigr)
\]
are finite and independent of the presentation of $R$ \cite{Lyubeznik}; finiteness of
the relevant Bass numbers is due to Huneke--Sharp in positive characteristic
\cite{HunekeSharp} and to Lyubeznik in characteristic zero \cite{Lyubeznik}. For
monomial ideals they admit a combinatorial description, and---unlike the numbers
persisted here---they genuinely depend on the characteristic of $\kk$ and detect
properties invisible to Hochster's formula alone
\cite{AlvarezMontanerVahidi}. The Alexander duality functors of Miller
\cite{Miller} and the squarefree-module formalism of Yanagawa \cite{Yanagawa}
provide the machinery in which such a persistent theory would most naturally be
developed. We regard persistent Bass and Lyubeznik numbers as the most substantial
open direction beyond this paper: they record information different both from the
present barcodes and from the Tor-side persistent Betti numbers of
\cite{SuwayyidWeiGraphs}.

\section*{Acknowledgment}
The work of KH was supported in part by the University of Georgia and Syracuse University. 
The work of FS was supported in part by the King Fahd University of Petroleum and Minerals. 
The work of GWW was supported in part by the University of Georgia, Georgia Research Alliance, and   NIH R35GM148196. 

\bibliographystyle{plain}

\bibliography{references}

@article{AlvarezMontanerSohrabi,
  author  = {J. {\`A}lvarez Montaner and F. Sohrabi},
  title   = {Bass numbers of local cohomology of cover ideals of graphs},
  journal = {J. Algebraic Combin.},
  volume  = {53},
  number  = {1},
  year    = {2021},
  pages   = {263--297}
}

@article{zia2025gbnl,
  title={{GBNL: Graded Betti} number learning of complex biological data},
  author={Zia, Mushal and Suwayyid, Faisal and Wei, Guo-Wei},
  journal={Foundations of Data Science},
  pages= {Doi: 10.3934/fods.2026014},
  year={2026}
}

@article{feng2025caml,
  title={CAML: Commutative Algebra Machine Learning A Case Study on Protein--Ligand Binding Affinity Prediction},
  author={Feng, Hongsong and Suwayyid, Faisal and Zia, Mushal and Wee, JunJie and Hozumi, Yuta and Chen, Chun-Long and Wei, Guo-Wei},
  journal={Journal of Chemical Information and Modeling},
  volume={65},
  number={13},
  pages={6732},
  year={2025}
}

@article{zhang2026commutative,
  title={Commutative Algebra Learning for Protein Flexibility Analysis},
  author={Zhang, Honghao and Feng, Hongsong},
  journal={arXiv preprint arXiv:2607.00879},
  year={2026}
}

@article{hu2025commutative,
  title={Commutative algebra-enhanced topological data analysis},
  author={Hu, Chuanshen and Wang, Yu and Xia, Kelin and Ye, Ke and Zhang, Yipeng},
  journal={arXiv preprint arXiv:2504.09174},
  year={2025}
}

@article{ren2025interpretability,
  title={Interpretability and Representability of Commutative Algebra, Algebraic Topology, and Topological Spectral Theory for Real-World Data},
  author={Ren, Yiming and Wei, Guo-Wei},
  journal={Advanced intelligent discovery},
  pages={e202500207},
  year={2025},
  publisher={Wiley Online Library}
}

@article{HelmMiller,
  author  = {D. Helm and E. Miller},
  title   = {Bass numbers of semigroup-graded local cohomology},
  journal = {Pacific J. Math.},
  volume  = {209},
  number  = {1},
  year    = {2003},
  pages   = {41--66}
}

@article{BubenikScott,
  author  = {P. Bubenik and J. A. Scott},
  title   = {Categorification of persistent homology},
  journal = {Discrete Comput. Geom.},
  volume  = {51},
  number  = {3},
  year    = {2014},
  pages   = {600--627}
}

@book{StanleyCCA,
  author    = {R. P. Stanley},
  title     = {Combinatorics and Commutative Algebra},
  edition   = {2nd},
  series    = {Progress in Mathematics},
  volume    = {41},
  publisher = {Birkh{\"a}user},
  year      = {1996}
}

@book{BrunsHerzog,
  author    = {W. Bruns and J. Herzog},
  title     = {Cohen--Macaulay Rings},
  series    = {Cambridge Studies in Advanced Mathematics},
  volume    = {39},
  publisher = {Cambridge University Press},
  year      = {1998}
}

@book{Eisenbud,
  author    = {D. Eisenbud},
  title     = {Commutative Algebra with a View Toward Algebraic Geometry},
  series    = {Graduate Texts in Mathematics},
  volume    = {150},
  publisher = {Springer},
  year      = {1995}
}

@book{MillerSturmfels,
  author    = {E. Miller and B. Sturmfels},
  title     = {Combinatorial Commutative Algebra},
  series    = {Graduate Texts in Mathematics},
  volume    = {227},
  publisher = {Springer},
  year      = {2005}
}

@article{Rahimi,
  author  = {A. Rahimi},
  title   = {Tameness of local cohomology of monomial ideals with respect to monomial prime ideals},
  journal = {J. Pure Appl. Algebra},
  volume  = {211},
  number  = {1},
  year    = {2007},
  pages   = {83--93},
  note    = {arXiv:math/0607256}
}

@article{EisenbudMustataStillman,
  author  = {D. Eisenbud and M. Musta{\c{t}}{\u{a}} and M. Stillman},
  title   = {Cohomology on toric varieties and local cohomology with monomial supports},
  journal = {J. Symbolic Comput.},
  volume  = {29},
  year    = {2000},
  pages   = {583--600}
}

@article{Grabe,
  author  = {H.-G. Gr{\"a}be},
  title   = {The canonical module of a Stanley--Reisner ring},
  journal = {J. Algebra},
  volume  = {86},
  year    = {1984},
  pages   = {272--281}
}

@article{BrunBrunsRomer,
  author  = {M. Brun and W. Bruns and T. R{\"o}mer},
  title   = {Cohomology of partially ordered sets and local cohomology of section rings},
  journal = {Adv. Math.},
  volume  = {208},
  year    = {2007},
  pages   = {210--235}
}

@article{AlvarezMontanerVahidi,
  author  = {J. {\`A}lvarez Montaner and A. Vahidi},
  title   = {Lyubeznik numbers of monomial ideals},
  journal = {Trans. Amer. Math. Soc.},
  volume  = {366},
  year    = {2014},
  pages   = {1829--1855},
  note    = {See also J. {\`A}lvarez Montaner, \emph{Lyubeznik table of sequentially Cohen--Macaulay rings}, Nagoya Math. J. \textbf{226} (2017)}
}

@book{ChazalDeSilvaGlisseOudot,
  author    = {F. Chazal and V. de Silva and M. Glisse and S. Oudot},
  title     = {The Structure and Stability of Persistence Modules},
  series    = {SpringerBriefs in Mathematics},
  publisher = {Springer},
  year      = {2016}
}

@article{CrawleyBoevey,
  author  = {W. Crawley-Boevey},
  title   = {Decomposition of pointwise finite-dimensional persistence modules},
  journal = {J. Algebra Appl.},
  volume  = {14},
  number  = {5},
  year    = {2015},
  pages   = {1550066}
}

@article{BotnanCrawleyBoevey,
  author  = {M. B. Botnan and W. Crawley-Boevey},
  title   = {Decomposition of persistence modules},
  journal = {Proc. Amer. Math. Soc.},
  volume  = {148},
  year    = {2020},
  pages   = {4581--4596},
  note    = {arXiv:1811.08946}
}

@article{wei2026commutative,
  author  = {G.-W. Wei},
  title   = {Commutative Algebra Meets Data Science: A New Paradigm in Mathematical Artificial Intelligence},
  journal = {Collections},
  volume  = {59},
  number  = {2},
  year    = {2026}
}

@article{suwayyid2025cakr,
  author  = {F. Suwayyid and Y. Hozumi and H. Feng and M. Zia and J. Wee and G.-W. Wei},
  title   = {{CAKR: Commutative} algebra k-mer representation of genomics},
  journal = {Nature Communication},
  year    = {2026},
  pages ={https://doi.org/10.1038/s41467-026-76429-z},
  url     = {https://doi.org/10.1038/s41467-026-76429-z}
}

@article{carlsson2009topology,
  title={Topology and data},
  author={Carlsson, Gunnar},
  journal={Bulletin of the American mathematical society},
  volume={46},
  number={2},
  pages={255--308},
  year={2009}
}

@article{SuwayyidWeiPSRT,
  author  = {F. Suwayyid and G.-W. Wei},
  title   = {Persistent {Stanley--Reisner} theory},
  journal = {Foundations of Data Science},
  volume  = {8},
  year    = {2026},
  pages   = {287--312},
  note    = {arXiv:2503.23482}
}

@article{SuwayyidWeiGraphs,
  author  = {F. Suwayyid and G.-W. Wei},
  title   = {Persistent commutative algebra on graphs and hypergraphs},
  journal = {Foundations of Data Science},
  year    = {2026},
  pages ={doi: 10.3934/fods.2026020},
  doi     = {10.3934/fods.2026020},
  note    = {arXiv:2512.17619, 2025}
}

@inproceedings{BendichCSEHM,
  author    = {P. Bendich and D. Cohen-Steiner and H. Edelsbrunner and J. Harer and D. Morozov},
  title     = {Inferring local homology from sampled stratified spaces},
  booktitle = {48th Annual IEEE Symposium on Foundations of Computer Science (FOCS'07)},
  year      = {2007},
  pages     = {536--546},
  doi       = {10.1109/FOCS.2007.33}
}

@inproceedings{BendichWangMukherjee,
  author    = {P. Bendich and B. Wang and S. Mukherjee},
  title     = {Local homology transfer and stratification learning},
  booktitle = {Proceedings of the 23rd Annual ACM-SIAM Symposium on Discrete Algorithms (SODA)},
  year      = {2012},
  pages     = {1355--1370}
}

@book{Munkres,
  author    = {J. R. Munkres},
  title     = {Elements of Algebraic Topology},
  publisher = {Addison-Wesley},
  address   = {Menlo Park, CA},
  year      = {1984}
}

@article{Lyubeznik,
  author  = {G. Lyubeznik},
  title   = {Finiteness properties of local cohomology modules (an application of $D$-modules to commutative algebra)},
  journal = {Invent. Math.},
  volume  = {113},
  number  = {1},
  year    = {1993},
  pages   = {41--55}
}

@article{HunekeSharp,
  author  = {C. L. Huneke and R. Y. Sharp},
  title   = {Bass numbers of local cohomology modules},
  journal = {Trans. Amer. Math. Soc.},
  volume  = {339},
  number  = {2},
  year    = {1993},
  pages   = {765--779}
}

@article{Miller,
  author  = {E. Miller},
  title   = {The {A}lexander duality functors and local duality with monomial support},
  journal = {J. Algebra},
  volume  = {231},
  number  = {1},
  year    = {2000},
  pages   = {180--234}
}

@article{Yanagawa,
  author  = {K. Yanagawa},
  title   = {{A}lexander duality for {S}tanley--{R}eisner rings and squarefree $\mathbb{N}^n$-graded modules},
  journal = {J. Algebra},
  volume  = {225},
  number  = {2},
  year    = {2000},
  pages   = {630--645}
}

\end{document}